\pdfoutput=1

\documentclass[12pt,english,reqno]{amsart}
\usepackage{amscd,amssymb,amsmath,amstext,amsthm,amsfonts}
\usepackage[dvips]{graphicx}
\usepackage{mathrsfs}
\usepackage{float}

\usepackage{a4wide}
\newcommand{\co}{\mbox{$\mathcal{O}$}}

\newcommand{\BQ}{\begin{quote}}
\newcommand{\EQ}{\end{quote}}

\newcommand{\ITEM}{\begin{description}}
\newcommand{\offITEM}{\end{description}}

\usepackage[pdftex]{hyperref}
\usepackage[usenames]{color}
\usepackage[ansinew]{inputenc}

\newtheorem{Theorem}{Theorem}
\newtheorem*{theorem}{Theorem}

\newtheorem{maintheorem}{Theorem}

\newtheorem{T}{Theorem}[section]
\newtheorem{Corollary}[T]{Corollary}
\newtheorem{Proposition}[T]{Proposition}
\newtheorem{Lemma}[T]{Lemma}

\newtheorem{Notation}[T]{Notation}
\newtheorem{Remark}[T]{Remark}

\newtheorem{Definition}[T]{Definition}

\newtheorem*{claim}{Claim}
\newtheorem{Claim}{Claim}

\def \BB {{\mathbb B}}
\def \RR {{\mathbb R}}
\def \ZZ {{\mathbb Z}}
\def \NN {{\mathbb N}}

\def \EE {{\mathbb E}}

\def \UU {{\mathbb U}}
\def \XX {{\mathbb X}}

\def \QQ {{\mathbb Q}}

\def \cl {\mathcal{L}}
\def \cf {\mathcal{F}}

\def \cp {\mathcal{P}}

\def \cu {\mathcal{U}}
\def \co {\mathcal{O}}
\def \cn {\mathcal{N}}
\def \ck {\mathcal{K}}

\def \cw {\mathcal{W}}

\def \cR {\mathscr{R}}

\newcommand{\dem}{\begin{proof}}
\newcommand{\cqd}{\end{proof}}

\newcommand{\leb}{\operatorname{Leb}}

\newcommand{\per}{\operatorname{Per}}
\newcommand{\period}{\operatorname{period}}
\newcommand{\interior}{\operatorname{interior}}

\newcommand{\diameter}{\operatorname{diameter}}

\begin{document}

\author{Paulo Brand\~ao}
\address{Impa, Estrada Dona Castorina, 110, Rio de Janeiro, Brazil.} \email{paulo@impa.br}

\date{\today}

\thanks{Partially supported by FAPERJ, CNPq and CAPES}

\title{Topological Attractors of Contracting Lorenz Maps}

\maketitle

\begin{abstract}

We study the non-wandering set of contracting Lorenz maps. We show that if such a map $f$ doesn't have any attracting periodic orbit, then there is a unique topological attractor. Precisely, there is a compact set $\Lambda$ such that $\omega_f(x)=\Lambda$ for a residual set of points $x \in [0,1]$. Furthermore, we classify the possible kinds of attractors that may occur.

\end{abstract}

\tableofcontents

\section{Introduction}

In \cite{Lor} Lorenz studied the solution of the system of differential equations~(\ref{eqlorenz}) in $\mathbb{R}^3$, originated by truncating the Navier-Stokes equations for modeling atmospheric conditions

\begin{eqnarray}
\label{eqlorenz}
\dot{x} & = & -10 x  + 10 y  \\
\dot{y} & = & 28 x  -y -xz \nonumber\\
\dot{z} & = & - \frac{8}{3} z  + x y  \nonumber 
\end{eqnarray}

\noindent
He observed what was thought to be an attractor with features that led to the present concept of a strange attractor. V.S. Afraimovich, V.V. Bykov, L.P. Shil'nikov, in \cite{ABS}, and Guckenheimer and  Williams, in \cite{GW}, introduced the idea of Lorenz-like attractors: dynamically similar models that also displayed the characteristics of the Lorenz strange attractor.

These models consist of a hyperbolic singularity with
one-dimensional unstable manifold such that, in a linearizable neighborhood, these separatrices can be considered as one of the coordinate axes, say $x$, in such a way that both components of $x\setminus \{0\}$ return to this neighborhood cutting transversally the plane $z = constant$, with the eigenvalues $\lambda_2 < \lambda_3< 0 < \lambda_1 $ (see Figure~\ref{LorenzFluxo}), and the expanding condition $ \lambda_3 + \lambda_1 > 0 $. We consider the Poincar\'e map of the square $Q=\left\{ \left| x\right| \leq cte;\left| y\right| \leq cte;z=cte\right\} $ into itself, having the returns as indicated in Figure~\ref{LorenzFluxo} and we can exhibit in $Q$ a foliation by one dimensional leaves, invariant by the Poincar\'e map, and such that it exponentially contracts the leaves. In \cite{GW} Guckenheimer and Williams show that given such a system, in a neighborhood $U$ the system is structurally  stable in $ codim \, 2 $, and in any representative family there is only a single attractor  attracting the neighborhood constructed.

\begin{figure}
  \begin{center}\includegraphics[scale=.40]{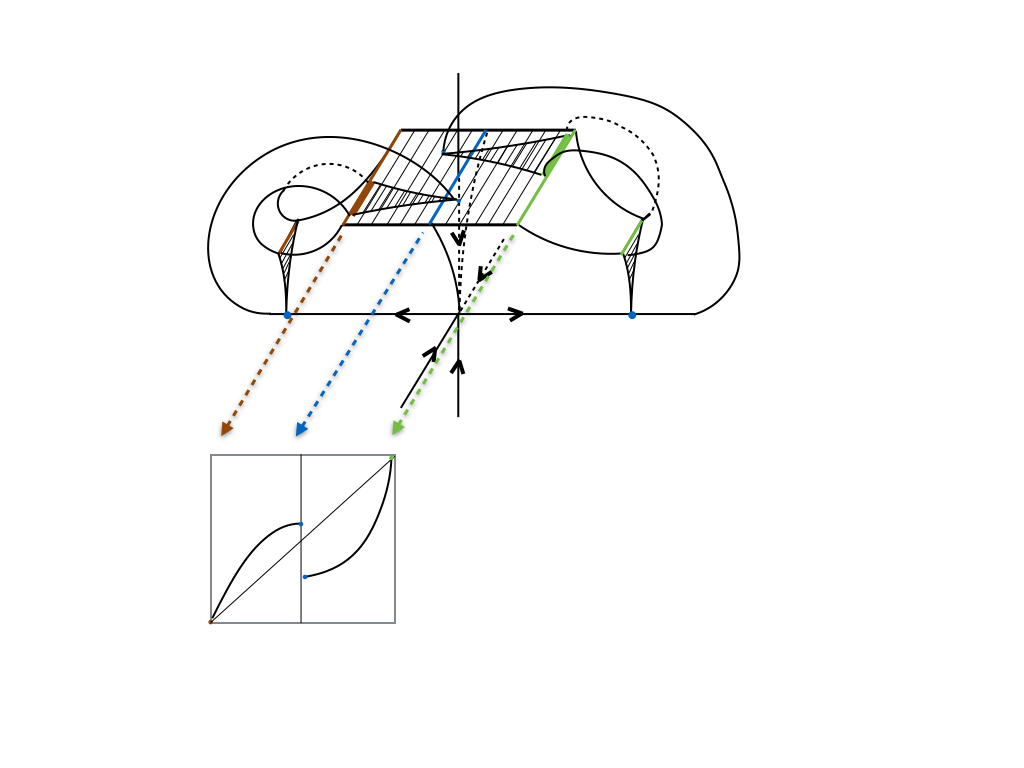}\\
  \caption{Lorenz-like Flow and Associated One-dimensional Dynamics}\label{LorenzFluxo}
  \end{center}
\end{figure}

In \cite{ACT}, Arneodo, Coullet and Tresser studied similar systems, just modifying the relation between the eigenvalues of the singularity, taking  $ \lambda_3 + \lambda_1 < 0 $: the so-called {\em contracting} Lorenz attractors. In this case the induced one-dimensional map is as displayed in Figure~\ref{LorenzFluxo}. 

Critical points and  critical values play fundamental roles in the study of dynamics of maps of the interval and from this point of view Lorenz maps are of hybrid type. Indeed, these maps have a single critical point, as unimodal maps do, but two critical values, as bimodal ones have. Because of this, it could perhaps happen that two different attractors would occur, but indeed we prove in Theorem~\ref{atratortopologico} that there is only one single topological attractor. That is, the behavior of contracting Lorenz maps looks like the one of unimodal maps, instead of the behavior of  bimodal maps, that admits up to two attractors. 

More specifically, we prove that, for contracting Lorenz maps, the possible long-term behavior scenarios for orbits of generic points are either periodic orbits, that only can be one or two of them, or a single attractor that can be one of the following types: cycle of intervals that forms a single chaotic attractor, Cherry attractor, Solenoid, or yet a subset of a chaotic Cantor set coexisting with wandering intervals. This last possibility, however, is expected not to occur, as conjectured by Martens and de Melo.

\section{Statement of the Main Results}\label{MainResults}
\label{mainresults}
We say an open interval $I$ is {\em of trivial dynamics} (up to some iterate) if $\exists n \in \NN$ such that $f^n|_I \equiv$ id.

\begin{Definition}[Lorenz Maps]
We say that a $C^2$ map $f:[0,1]\setminus \{c\} \rightarrow[0,1]$, $0<c<1$, is a {\em Lorenz map} if $f(0)=0$, $f(1)=1$, $f'(x)>0$, $\forall\,x\in[0,1]\setminus \{c\}$. A Lorenz map is called {\em contracting} if $\lim_{x\to c}f'(x)=0$ and there is no interval of trivial dynamics.
\end{Definition}

Given $n\ge1$, define $f^{n}(c_{\pm})=\lim_{0<\epsilon \to 0}f^{n}(c \pm \epsilon)$. The critical values of $f$ are $f(c_{-})$ and $f(c_{+})$. If $x\in\{f(c_{-}),f(c_{+})\}$ set $f^{-1}(x)=\{c\}\cup\{y\in[0,1]\,;\,f(y)=x\}$. Given a set $X\subset[0,1]$, define $f^{-1}(X)=\bigcup_{x\in X}f^{-1}(x)$. Inductively, define $f^{-n}(x)=f^{-1}(f^{-(n-1)}(x))$, where $n\ge2$.
The {\em pre-orbit} of a point $x\in[0,1]$ is the set $\co_{f}^{-}(x):=\bigcup_{n\ge0}f^{-n}(x)$, where $f^{0}(x):=x$. Denote the positive orbit of a point $x\in[0,1]\setminus\co_{f}^{-}(c)$ by $\co_f^+(x)$, i.e., $\co_{f}^{+}(x)=\{f^j(x);j\ge0\}$. If $\exists p \ge 1$ such that $f^{p}(c_{-})=c$, we take $p \in \NN$ as being minimal with this property and define $\co_{f}^+(c_{-})=\{f^{j}(c_{-})\,;\, 1\le j \le p\}$. Otherwise, if $\nexists p \ge 1$ such that $f^{p}(c_{-})=c$, we define $\co_{f}^+(c_{-})=\{f^{j}(c_{-})\,;\,j\ge0\}$. Similarly we define $\co_{f}^+(c_{+})$. If $x\in\co_{f}^{-}(c)$, let $\co_{f}^{+}(x)=\{x,f(x),\cdots,f^{m_{x}-1}(x),c\} \cup  \co_{f}^+(c_{-})  \cup   \co_{f}^+(c_{+})  $, with $m_x$ minimum such that $f^{m_{x}}(x)=c$. Also, $\co_{f}^{+}(X)$ denotes the positive orbit of $X$ by $f$, that is,
$\co_{f}^{+}(X)=\bigcup_{x \in X} \co_f^+(x)  $.

A point $x$ is said to be {\em non-wandering} if for any neighborhood $U \ni x$, $\exists n \ge 1$ such that $f^n(U)\cap U\ne \emptyset$. The set of all non-wandering points is the {\em non-wandering set} $\Omega(f)$. The set of accumulation points of the positive orbit of $x\in[0,c)\cup\{c_{-},c_{+}\}\cup(c,1]$ is denoted by $\omega_f(x)$, the $\omega$-limit set of $x$.  The $\alpha$-limit set of $x$, $\alpha_{f}(x)$, is the set of points $y$ such that $y=\lim_{j\to\infty}x_{j}$ for some sequence $x_{j}\in f^{-n_{j}}(x)$ with $n_{j}\to+\infty$.

Following Milnor \cite{Milnor:1985ut}, a  compact set $A$ is a {\em topological attractor} if 
its basin $\beta(A) = \{x; \omega_{f}(x) \subset A\}$ is residual in an open set and if each closed forward invariant subset $A'$ which is strictly contained in $A$ has a topologically smaller basin of attraction, i.e., $\beta(A) \setminus \beta(A')$ is residual in an open set. (Similarly, $A$ is a {\em metrical attractor} if $\leb \beta(A)> 0$ and $\leb \big(\beta(A) \setminus \beta(A')\big)>0$, $\forall A'$ closed forward invariant $A'\subsetneq A$).

Given a periodic point $p$, say $f^n(p)=p$, we say that its periodic orbit $\co_f^+(p)$ is an {\em attracting periodic orbit} if $\exists \epsilon >0$ such that $(p,p+\epsilon)$ or $(p-\epsilon,p)\subset \beta(\co_f^+(p))$. A {\em periodic attractor} is a finite set $\Lambda$ such that $\interior(\{x\,;\,\omega_{f}(x)=\Lambda\})\ne\emptyset$, and it can be either an attracting periodic orbit, or a {\em super-attractor}: a finite set $\Lambda=\{p_{1},\cdots,p_{n},c\}$ such that $f(p_{i})=p_{i+1}$ for $1\le i<n$, $f(p_{n})=c$ and $\lim_{0<\varepsilon\downarrow0}f(c+\varepsilon)=p_{1}$ or $\lim_{0<\varepsilon\downarrow0}f(c-\varepsilon)=p_{1}$. A {\em weak repeller} is a periodic point $p$ of $f$ such that it is non-hyperbolic and it is not a periodic attractor.

We say $I$ is {\em a wandering interval} of $f$ if $f^n|_I$ is a homeomorphism for $\forall n\ge1$, $f^i(I)\cap f^j(I)=\emptyset$ for $i\ne j >0$ and $I$ doesn't intersect the basin of an attracting periodic orbit.

We say that an attractor (topological or metrical) $\Lambda$ is a {\em chaotic attractor} if $\Lambda$ is transitive,  periodic orbits are dense in it ($\overline {Per(f)\cap \Lambda}=\Lambda$), its topological entropy $h_{top}(f|_\Lambda)$ is positive and $\exists \lambda >0$ and a dense subset of points $x \in \Lambda$ such that their {\em Lyapounov exponents}, $\exp_f(x)$, are greater than $\lambda$, where $\exp_f(x):=\liminf \frac{1}{n}\log |Df^n(x)|$.

A {\em cycle of intervals} is a transitive finite union of non-trivial disjoint closed intervals.

A {\em gap map} is a continuous and injective map $g:S^1\setminus\{c\}\to S^1$, where $S^1=\RR/\ZZ$ is the circle and $c$ is any point of it. It is a known fact that such a map has a well defined rotation number $\rho(g)$. Furthermore, if $\rho(g)\notin\QQ$, then $g$ is semi-conjugated to an irrational rotation. In this case there exists a minimal set $\Lambda$ containing $c$ such that $\omega_g(x)=\Lambda$ for every $x\in S^1$ (if $x\in\bigcup_{j\ge0}g^{-j}(c)$ we consider $\omega_g(x_{\pm})$ instead of $\omega_g(x)$). 

We say that a Lorenz map $f$ is a {\em Cherry map} if there is a neighborhood $J$ of the critical point such that the first return map to $J$ is conjugated to a gap map with an irrational rotation.
It follows from \cite{GT85} that a Lorenz map $f$ is a Cherry map if and only if $f$ does not admit super-attractors and there exists a neighborhood $J$ of the critical point $c$ such that $c\in\omega_f(x_{\pm})$, $\forall\,x\in J$. If $f$ is a Cherry map, $\Lambda:=\omega_f(c_{-})=\omega_f(c_+)$ is called a {\em Cherry attractor} and it is a minimal compact set containing the critical point $c$ in the interior of its basin of attraction.

A {\em renormalization interval} for $f$ is an open interval $J=(a,b)\ni c$ such that the first return map to $[a,b]$  is conjugated to a Lorenz map. Their points of boundary are always periodic points and $f^{\period(a)}([a,c))\subset[a,b]\supset f^{\period(b)}((c,b])$. Further properties of intervals of this type  will be studied in Section \ref{renormecherry}.

Given a renormalization interval $J=(a,b)$, define the {\em renormalization cycle} associated to $J$ (or generated by $J$) as $$U_J=\bigg(\bigcup_{i=0}^{\period(a)}f^{i}((a,c))\bigg)\cup\bigg(\bigcup_{i=0}^{\period(b)}f^{i}((c,b))\bigg).$$

Given $J\subset[0,1]$ an open set with $c\in J$, define $\Lambda_{J}:=\{x\in [0,1] \,;\,\co^{+}_{f}(x)\cap J=\emptyset\}$. We call a {\em gap of $\Lambda_{J}$} any connected component of $[0,1]\setminus \Lambda_J$. We also define the set $K_J$, the {\em nice trapping region associated to J}, as being the set formed by the union of gaps of $\Lambda_J$ such that each of these gaps contains one interval of the renormalization cycle.

We say that $f$ is {\em $\infty$-renormalizable} if $f$ has infinitely many different renormalization intervals. An attractor $\Lambda$ of a contracting Lorenz map $f$ is a {\em Solenoidal attractor} (or {\em Solenoid})  if  $\Lambda\subset\bigcap_{n=0}^\infty K_{J_n}$, and $\{J_n\}_{n}$ is an infinite nested chain of renormalization intervals.

A Contracting Lorenz map $f:[0,1]\setminus\{c\} \to [0,1]$ is called {\em non-flat} if there exist constants $\alpha$,$\beta>1$, $a,b\in[0,1]$ and $C^2$ orientation preserving  diffeomorphisms $\phi_{0}:[0,c]\to[0,a^{1/\alpha}]$ and $\phi_{1}:[c,1]\to[0,b^{1/\beta}]$ such that $$f(x)=\begin{cases}a-(\phi_{0}(c-x))^\alpha&\text{ if }x<c\\
1-b+(\phi_{1}(x))^\beta&\text{ if }x>c\end{cases}.$$

\begin{maintheorem}[The Solenoid attractor]\label{SOLENOIDETH}
Let $f$ be a $C^{2}$ non-flat contracting Lorenz map without periodic attractors. If $f$ is $\infty$-renormalizable, then there is a compact minimal set $\Lambda$, with $c\in\Lambda\subset\bigcap_{J\in\cR }K_J$ such that $\omega_f(x)=\Lambda$, $\forall\,x\in[0,1]$ with $c\in\omega_f(x)$, where $\cR $ is the set of renormalization intervals $J$ of $f$ and $K_J$ their corresponding nice trapping regions. 
\end{maintheorem}

\begin{maintheorem}
\label{baciastopologicas}

If $f$  is a $C^{2}$ non-flat contracting Lorenz map without periodic attractors, then $f$ has a transitive topological attractor $\Lambda$. Furthermore, $\beta(\Lambda)$ is a residual set in the whole interval, and $\Lambda$ is one and only one of the following types:

\begin{enumerate}

\item {\em Cherry attractor} and in this case $\omega_{f}(x)=\Lambda$ in an open and dense set of points $x\in [0,1]$. 
\item {\em Solenoidal attractor} and in this case $\omega_{f}(x)=\Lambda$ in a residual set of points $x\in [0,1]$. 
\item {\em Chaotic attractor} that can be of two kinds:
\begin{enumerate}
\item {\em Cycle of intervals}, in this case $\omega_{f}(x)=\Lambda$ in a residual set of points $x\in [0,1]$. 
\item {\em Cantor set} and in this case there are wandering intervals. 

\end{enumerate}

\end{enumerate}

\end{maintheorem}

\begin{maintheorem}
\label{teoalfalim}
Let $f$ be a $C^{2}$ non-flat contracting Lorenz map without periodic attractors and $\Lambda$ its single topological attractor as obtained in Theorem \ref{baciastopologicas}. Then, $f$ has no wandering interval if and only if $\alpha_{f}(x)=[0,1], \forall x \in \Lambda$.

\end{maintheorem}

The next theorem goes deeper in the classification provided by Theorem~\ref{baciastopologicas}, as it distinguishes two possible situations for item (3)(b) of that theorem. Observe that item (3)(b) didn't state that the Cantor set $\Lambda$ is equal to $\omega_{f}(x)$ for a residual set of $x \in [0,1]$, but only that the basin $\beta(\Lambda)$ contains a residual set. That is, (3)(b) can split into two situations. In the first one, $\Lambda$ attracts a residual set whose $\omega$-limit coincides with $\Lambda$. In the case this doesn't happen, under some additional hypothesis we can have that $\Lambda$ properly contains another Cantor set $\Lambda'$ such that its basin $\beta(\Lambda')$ is residual in $[0,1]$ and $\forall x \in \beta(\Lambda')$ is such that $\omega(x)=\Lambda'$.

We say that a $C^3$ map $f$  has negative Schwarzian derivative, denoted by $Sf$, if $Sf$ is negative in every point $x$ such that $Df(x)\ne0$, where

\begin{equation}\label{defschwarz}
Sf(x)=\frac{D^3 f(x)}{D f(x)}-\frac{3}{2}\bigg( \frac{D^2 f(x)}{D f(x)} \bigg)^2
\end{equation}

\begin{maintheorem}
\label{atratortopologico}
Let $f$ be a $C^3$ non-flat contracting Lorenz map with negative Schwarzian derivative.
If $f$ has a periodic attractor $\Lambda$, then either $\beta(\Lambda)$ is an open and dense set or there is another periodic attractor $\Lambda'$ such that $\beta(\Lambda)\cup\beta(\Lambda')$ is open and dense.

If $f$ does not have any periodic attractor, then there is a single topological attractor $\Lambda$ with $\omega_{f}(x)=\Lambda$ for a residual set of points $x \in [0,1]$ and it is one of the following types:

\begin{enumerate}
\item {\em $\Lambda$ is a Cherry attractor}; 
\item {\em $\Lambda$ is a solenoidal attractor}; 
\item {\em $\Lambda$ is a chaotic cycle of intervals};
\item {\em $\Lambda=\overline{\co_f^+(c_{+}) \cup\co_f^+(c_{-})  }$ and it is contained in a chaotic Cantor set whose gaps are wandering intervals}
\end{enumerate}

\end{maintheorem}

Theorem \ref{atratortopologico} allows us to compare between the metrical and topological attractors. Indeed we can conclude that (1) the topological attractor contains the metrical one, and (2) If the topological attractor is not a cycle of intervals, then the topological attractor and the metrical one coincide. The existence and classification of metrical attractors can be found in \cite{StP}.

Results on contracting Lorenz maps and flows date from the beginning of the 1980's. In this decade and the first half of the 1990's, we mention C. Tresser, A. Arneodo, L. Alsed\`a, A. Chenciner, P. Coullet, J-M. Gambaudo, M. Misiurewicz, A. Rovella, R.F. Williams (see \cite{ACT,CGT84,GT85,Gambaudo:1986p2422,Tresser:1993uf,Rov93}).
Later on, main contributions include  M. Martens and W. de Melo \cite{MM}, G. Keller and M. St. Pierre \cite{Keller:2011p1585,StP}, D. Berry and B. Mestel \cite{BM}, and R. Labarca and C. G. Moreira \cite{Labarca:2010p1505,Labarca:2006p1486}.

\section{Preliminary Results}

A {\em homterval} is an open interval $I=(a,b)$ such that $f^n|_I$ is a homeomorphism for $n\ge1$ or, equivalently, $I\cap\co_f^-(c)=\emptyset$. 

Let us denote by $\BB_0(f)$ the union of the basins of attraction of all periodic attractors of $f$.

\begin{Lemma}[Homterval Lemma, see \cite{MvS}]\label{homtervals}
Let $f:[0,1]\setminus\{c\}\to[0,1]$ be a $C^{2}$ non-flat contracting Lorenz map and $I=(a,b)$ be a homterval of $f$. If $I$ is not a wandering interval, then $I\subset\BB_0(f)\cup\co_f^{-}(\per(f))$. Furthermore, if $f$ is $C^3$ with $Sf<0$, and $I$ is not a wandering interval, then the set $I\setminus\BB_0$ has at most one point.
\end{Lemma}

\begin{Lemma}\label{LemmaWI}If $f:[0,1]\setminus\{c\}\to[0,1]$ is a $C^{2}$ non-flat contracting Lorenz map, then every wandering interval accumulates on both sides of the critical point. In particular, a wandering interval cannot contain any interval of the form $(-r,c)$ or $(c,r)$.
\end{Lemma}

\dem
Suppose we have a wandering interval $J$ that doesn't accumulate on the right side of the critical point, say, it never enters a neighborhood $(c,c+\varepsilon)$. So, we can modify $f$ to coincide with the original function outside this interval, but being $C^2$ and non-flat in this interval (see Figure~\ref{errantepng}). In this way, the modified function is a  $C^{2}$ map displaying a wandering interval, but it is a known fact that this can't happen with a $C^2$ map with non-flat critical points (see Theorem A, Chapter IV of \cite{MvS}).

\begin{figure}
\begin{center}
\includegraphics[scale=.2]{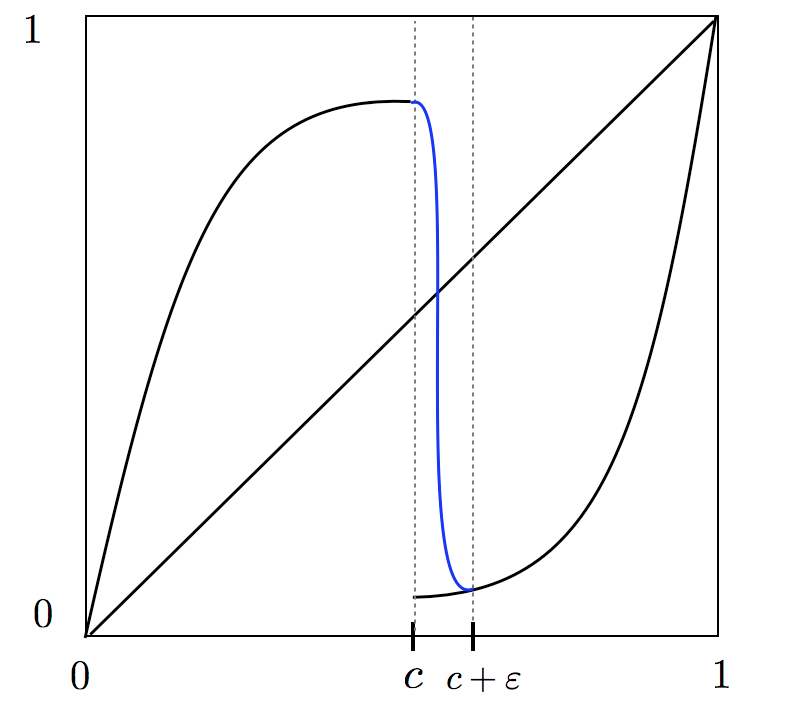}\\
\caption{}\label{errantepng}
\end{center}
\end{figure}

\cqd

One can adapt the well known Singer's Theorem to our context, with $f:[0,1]\setminus\{c\}\to[0,1]$ being a $C^{3}$ non-flat contracting Lorenz map with negative Schwarzian derivative, and obtain that the immediate basin of any attracting periodic orbit of this map contains in its border either its critical point or a boundary point of  $[0,1]$. From this we obtain that $f$ can have, at most, two attractors of periodic type (one can also obtain that each neutral periodic point is an attracting periodic orbit and that there exists no interval of periodic points). We can go even further and state:

\begin{Proposition}
\label{DOISMAX}
Let $f:[0,1]\setminus\{c\}\to[0,1]$ be a $C^{3}$ non-flat contracting Lorenz map with negative Schwarzian derivative. Then $f$ can have at most two periodic attractors and, when it has a periodic attractor, the union of the basins of the periodic attractors is always an open and dense set. 
\end{Proposition}

\dem

Let $p$ be so that $\co_{f}^{+}(p)$ is an attracting periodic orbit. Notice that $\beta(\co_{f}^{+}(p))$ is an open set. By Singer's theorem $(c, \delta) \subset \beta(\co_{f}^{+}(p))$ for some $\delta >0$ (or $(-\delta, c) \subset \beta(\co_{f}^{+}(p))$, which is similar). If $\overline{\beta(\co_{f}^{+}(p))} \ne [0,1]$, then there is a connected component  $T$ of $[0,1] \setminus \overline{\beta(\co_{f}^{+}(p))}$.

If $\exists j$ such that $f^j(T)\ni c$, $\exists y \in T$ such that $f^j(y)\in (c,\delta)$, then $y\in \beta(\co_{f}^{+}(p))$, leading to an absurd, as $y \in T$ and $T \subset [0,1] \setminus \overline{\beta(\co_{f}^{+}(p))}$.

In this way, for any given $j$, $f^j|_T$ is a homeomorphism, so $T$ is a homterval and then it is either a wandering interval or it intersects the basin of attraction of an attracting periodic orbit that can't be $\co_{f}^{+}(p)$ as $ T \subset [0,1] \setminus \overline{\beta(\co_{f}^{+}(p))}$. 

The first case can't occur, as $T$ cannot be a wandering interval, as its orbit would accumulate in $c$ by both sides (by Lemma \ref{LemmaWI}) and then there would be $j$ such that $f^j(T)\subset (c,\delta)$, leading again to an absurd. In the second case, $\exists q$ such that $\co_{f}^{+}(q)$ is an attracting periodic orbit, and $\co_{f}^{+}(q) \ne \co_{f}^{+}(p)$. Finally, if $\overline{\beta(\co_{f}^{+}(p))} \cup \overline{\beta(\co_{f}^{+}(q))}\ne [0,1]$, there would be a connected component of $[0,1] \setminus ( \overline{\beta(\co_{f}^{+}(p))} \cup \overline{\beta(\co_{f}^{+}(q))} )$ and we could show in the same way it is a homterval, that cannot be wandering. Also, it cannot be in the basin of a third periodic orbit, as this would have to have the critical point in its border, but both sides of it are already attracted to one or possibly two aforementioned orbits. 

\cqd

\begin{Lemma}\label{Remark98671oxe}Let $f:[0,1]\setminus\{c\}\to[0,1]$ be a $C^{2}$ contracting Lorenz map. If $f$ does not have any periodic attractor, then there is a residual set $U$ such that $$c\in\omega_{f}(x)\, \forall x \in U.$$
Furthermore, given any neighborhood $V$ of the critical point, the set of points that visit $V$ is an open and dense set.
\end{Lemma}

\dem Let $J_n=\{x\in[0,1]\|\co^+_f(x)\cap(c-1/n,c+1/n)\ne\emptyset\}$, $n\in \NN$, so $J_n$ is open and non-empty. If $J_n$ was not dense, then $\exists (a,b)\subset [0,1]\setminus \overline J_n$. By the homterval lemma, as $f$ has no periodic attracting orbit, there would be $\ell \ge 0$ such that $f^\ell\big( (a,b)\big) \ni c$ or $(a,b)$ would be a wandering interval. The first case would imply that $(a,b)\cap \overline J_n\ne \emptyset$. The second one also cannot happen, as otherwise iterates of $(a,b)$ would approach $c$,  by Lemma \ref{LemmaWI}, and this would lead to the same contradiction.
Then, $J=\cap_{n\ge 0}J_n$ is residual and we have that $c\in \omega(x)$, $\forall x \in J$.

\cqd

A metrical version of this lemma also can be obtained as a consequence of \cite{Man85} if we add the hypothesis that the map has no weak repellers.

\begin{theorem}[Koebe's Lemma \cite{MvS}] 

For every $\varepsilon>0$, $\exists K>0$ such that the following holds: let $M$, $T$ be intervals in $[0,1]$ with $M\subset T$ and denote respectively by $L$ and $R$ the left and right components of $T\setminus M$ and let $f:[0,1]\to[0,1]$ be a map with negative Schwarzian derivative. If $f^{n}|_T$ is a diffeomorphism for a given $n\ge1$ and 
$$ |f^{n}(L)|\ge\varepsilon|f^{n}(M)| \text{ and } |f^{n}(R)|\ge\varepsilon|f^{n}(M)|,$$ then $\frac{|Df^{n}(x)|}{|Df^{n}(y)|}\le K$ for $x,y\in M$.

\end{theorem}

\begin{Lemma} \label{trescinco}Let $f:[0,1]\setminus\{c\}\to[0,1]$ be a $C^{3}$ non-flat contracting Lorenz map with negative Schwarzian derivative. If $I$ is a wandering interval, $\forall y \in I$, $\omega_f(y)= \overline{\co_f^+(c_{+}) \cup\co_f^+(c_{-})  }$.
\end{Lemma}

\dem

We have shown in Lemma \ref{LemmaWI} that 
the orbit of any given wandering interval $I$ accumulates in the critical point by both sides, and then, by continuity we have  
$\omega_f(I)\supset \overline{\co_f^+(c_{+}) \cup\co_f^+(c_{-})  }$. 
Now, suppose there is $p \in \omega_f(I)$ such
that $p \not\in \overline{\co_f^+(c_{+}) \cup\co_f^+(c_{-})  }$.
We can also suppose without loss of generality
that $I$ is maximal, in the sense that there is no bigger wandering interval that contains $I$ properly.
Let $T$ be a connected component of $[0,1] \setminus \big(\overline{\co_f^+(c_{+}) \cup\co_f^+(c_{-})  }\big)$ containing $p$.
Given $\epsilon>0$, let $n_{\epsilon}$ be the
minimum $j$ such that $f^j (I) \subset B_{\epsilon}(p)$.
Let $T_{\epsilon}$ be the maximal interval
containing I such that $f^{n_{\epsilon}}(T_{\epsilon}) \subset T$
 and that $f^{n_{\epsilon}}|_{T_{\epsilon}}$ is a diffeomorphism.

Notice that $f^{n_{\epsilon}}(T_{\epsilon})=T$ for otherwise, there would exist $y \in T$ such that
$y= f^{n_{\epsilon}}(a)$, where $a \in \partial T_{\epsilon}$. And as $f^{n_{\epsilon}}|_{T_{\epsilon}}$ cannot be monotonously extended to a bigger interval, then $\exists 0 \le j < {n_{\epsilon}}$ such that $f^j(a) =c$,
which would lead to an absurd, as
$f^{({n_{\epsilon}-j})}(a_+) \in T  \subset [0,1] \setminus\overline{\co_f^+(c_{+}) \cup\co_f^+(c_{-})  }$ (or this would occur for $f^{({n_{\epsilon}-j})}(a_-)$).

Let $J_{\epsilon}=(f^{n_{\epsilon}}|_{T_{\epsilon}}) ^{-1}    (B_{\epsilon}(p))$
and $U = \cap_{\epsilon>0} J_{\epsilon}$.
As ${\epsilon} \to 0$ implies ${n_{\epsilon}}\to \infty$, and as every $f^{n_{\epsilon}}$ is a diffeomorphism onto its image, $\forall{n_{\epsilon}}$, it follows that $f^j$ is a diffeomorphism in $U$, $\forall j$. In this way, $U$ is a homterval and then either  $U$ is a wandering interval or $U \in \co ^-(Per(f)) \cup \BB_0(f)$.
As $U \supset I$ it cannot be as in the second case for $I$ being wandering implies there is no periodic attractor, and as $I$ was taken as maximal, we have necessarily that $U=I$. 

We can take $\epsilon_0$ small enough such that the left and right connected components of $T\setminus 
B_{\epsilon}(p)$ are as big as we want compared to $|B_{\epsilon}(p)|$, in such a way that Koebe's Lemma ensures that given any $\epsilon > 0$ such that $\epsilon < \epsilon_0$, $\exists K >0$ such that $\frac{|Df^{n_\epsilon}(x)|}{|Df^{n_\epsilon}(y)|}\le K$, $\forall x, y \in J_\epsilon$, $\forall \epsilon \in (0,\epsilon)$. Recall that $B_{\epsilon}(p)=f^{n_{\epsilon}}(J_{\epsilon})$ and $|f^{n_{\epsilon}}(J_{\epsilon})|\ge (1/K) m |J_\epsilon|$ where $m=|Df^{n_\epsilon}(x_0)|$, to some $x_0 \in M$ and also $|f^{n_{\epsilon}}(J_\epsilon\setminus I)|\le K m |J_\epsilon \setminus I|$. So, we have the following inequality

$$
\frac{|B_{\epsilon}(p) \setminus f^{n_{\epsilon}}(I)|}{|B_{\epsilon}(p)|}=
\frac{|f^{n_{\epsilon}}(J_\epsilon\setminus I)|}{|f^{n_{\epsilon}}(J_\epsilon)|}\le
K^2 \frac{ |J_\epsilon\setminus I|}{|J_\epsilon|}<1/2.
$$

The last inequality follows from the fact that the collection of $J_\epsilon$ cannot have subsequences whose limit would be bigger than $I$, for otherwise the intersection of them would generate a bigger wandering interval, in contradiction to the maximality of $I$. So, we can calculate these estimates on a nested subsequence of $J_\epsilon$ whose intersection is $I$, and so we can take $\epsilon$ small enough such that 
$$
\frac{|J_\epsilon\setminus I|}{|J_\epsilon|}<\frac{1}{2 K ^2}
$$
then
$$
\frac{|f^{n_{\epsilon}}(I)|}{|B_{\epsilon}(p)|}>1/2
$$
and then $p \in f^{n_{\epsilon}}(I)$, which is a contradiction, as $p$ was chosen as belonging to $\omega_f(I)$ where $I$ is a wandering interval.

\cqd

\begin{Lemma}[Denseness of wandering intervals, when they exist]\label{DenWanInt}
Let $f:[0,1]\setminus\{c\}\to[0,1]$ be a $C^{2}$ non-flat contracting Lorenz map without periodic attractors.  If $f$ has a wandering interval $I$, then $\cw$ is an open and dense set, where $\cw$ is the union of all open wandering intervals of $f$.
\end{Lemma}
\begin{proof}
If $\cw$ is not dense, then $[0,1]\setminus\overline{\cw}$ contains some open interval $I$. Clearly, $I$ is not a wandering interval. As $f$ does not have periodic attractors, we can apply Lemma~\ref{homtervals} and conclude that there is $n\in\NN$ such that $f^{n}|_{I}$ is a homeomorphism and that $f^{n}(I)\ni c$. As $\cw$ is invariant ($f^{-1}(\cw)=\cw$), $[0,1]\setminus\overline{\cw}$ is also invariant. Thus, $c\in\interior([0,1]\setminus\overline{\cw})$, that is, there is no wandering interval in a neighborhood of $c$. And this is not possible, by Lemma~\ref{LemmaWI}.

\end{proof}

\begin{Corollary}
\label{stpremovido}
Let $f:[0,1]\setminus\{c\}\to[0,1]$ be a $C^{3}$ non-flat contracting Lorenz map with negative Schwarzian derivative displaying no periodic attractors.
If $f$ has a wandering interval $I$, then there is an open and dense set $U$ such that any given $x \in U$, $\omega_f(x) = \overline{\co_f^+(c_{+}) \cup\co_f^+(c_{-})  }$.

\end{Corollary}

\dem
Taking $U$ as the set $\cw$ of Lemma~\ref{DenWanInt}, $U$ satisfies the required condition by applying Lemma \ref{trescinco}.
\cqd

\section{Periodic Points}\label{Section98767}

Given an interval $J=(a,b)$ and a map $f$ defined in $J$, denote the {\em first return map} to $J$ by $\cf_{J}:J^{*}\to J$. That is, $\cf_{J}(x)=f^{R(x)}(x)$, where $J^{*}=\{x\in J$ $;$ $\co^{+}_{f}(f(x))\cap J\ne\emptyset\}$ and $R(x)=\min\{j\ge1$ $;$ $f^{j}(x)\in J\}$, that is called the {\em first return time}. Let $\cp_{J}$ be the collection of connected components of $J^{*}$.

An open interval $I=(a,b)$ containing the critical point $c$
is called a  {\em nice interval} of $f$ if
$\co_{f}^{+}(\partial I)\cap I=\O$ and $a$ and $b\not\in \beta(\co^+(p))\setminus\co^+(p)$, $p$ a periodic attractor. We will denote the {\em set of nice intervals of $f$} by
$\mathcal{N}=\mathcal{N}(f)$ and the set of nice intervals whose
borders belong to the set of periodic points of $f$ by
$\mathcal{N}_{per}=\mathcal{N}_{per}(f)$, that is,
$\mathcal{N}_{per}=\{I\in\mathcal{N}\ \|\ \partial I\subset Per(f)\}$.

\begin{Lemma}\label{Lemma8388881a}
Let $f:[0,1]\setminus\{c\}\to[0,1]$ be a $C^{2}$ non-flat contracting Lorenz map and let $J=(a,b)$ be a nice interval, with first return map $\cf_{J}:J^{*}\to J$.
The following statements are true.
\begin{enumerate}
\item $\big((p,q)\in\cp_{J}\text{ and }p\ne c\big) \Rightarrow \cf_{J}((p,q))=(a,f^{R|_{(p,q)}}(q));$
\item $\big((p,q)\in\cp_{J}\text{ and }q\ne c\big) \Rightarrow \cf_{J}((p,q))=(f^{R|_{(p,q)}}(p),b);$
\item $\big(I\in\cp_{J}\text{ and }c\notin\partial I\big) \Rightarrow \cf_{J}(I)=J.$
\end{enumerate}
\end{Lemma}
\dem\,
Assume that $I=(p,q)\in\cp_{J}$ and $p\ne c$. Let $n=R|_{I}$. 

If $p=a$, then 
 
(i) If $f^{n}(p)<a$, then $f^{n}(p+\varepsilon)<a$ for $\varepsilon>0$ sufficiently small. This is an absurd, as $n$ is a return time of $p+\varepsilon\in I$. 
 
(ii) If $f^{n}(p)\ge b$, as $f$ preserves orientation, $f^{n}(p+\varepsilon)\ge b$, that will also be in contradiction with the fact that $n$ is a return time of $(p+\varepsilon)\in I$.
 
(iii) $f^{n}(p)\in(a,b)$ also leads to a contradiction, because $J$ is nice. So, $f^{n}(p)=a$ whenever $p=a$.

Consider now $a<p$ and $p\ne c$. Cases (i) and (ii) can be proved as before, and the remaining case, if $f^{n}(p)\in(a,b)$, $\exists \varepsilon$ sufficiently small such that, $(p,p+\varepsilon)$ doesn't return until $n$, as $n$ is the first return time of $I$ to $(a,b)$, $f^{j}(I)\cap (a,b)=\emptyset$ for every $0<j<n$.
Thus,  $f^{j}(p)\ne c$, $\forall\,0\le j<n$.
Thus, $f^{n}$ is continuous in $(p-\delta,p+\delta)$ for a sufficiently small $\delta>0$.
As a consequence, if $a<f^{n}(p)<b$, then, taking $\delta>0$ small, $n$ will be the first return time for $(p-\delta,q)$ to $(a,b)$, contradicting $I\in\cp_{J}$.
So, we necessarily have $f^{n}(p)=a$, proving (1).

Similarly, (2) follows from the same kind of reasoning, and (3) is a consequence of (1) and (2)
\cqd

\begin{Corollary}\label{Corollary8388881b}
Let $f:[0,1]\setminus\{c\}\to[0,1]$ be a $C^{2}$ non-flat contracting Lorenz map and let $J=(a,b)$ be a nice interval, with first return map $\cf_{J}:J^{*}\to J$.
If $J=(a,b)$ is a nice interval and $f$ is a contracting Lorenz map defined in $J$, then the following statements are true:
\begin{enumerate}
\item $a\in\partial I\text{ for some }I\in\cp_{J}\Leftrightarrow\,\,a\in Per(f).$
\item $b\in\partial I\text{ for some }I\in\cp_{J}\Leftrightarrow\,\,b\in Per(f).$
\end{enumerate}
\end{Corollary}
\dem
If $I=(a,q)\in\cp_{J}$ (the case $I=(q,b)$ is analogous) and $n=R|_{I}$, it follows from Lemma~\ref{Lemma8388881a} that 
$f^n(a)=\cf_{J}(a)=a$. That is, $a$ is a periodic point.

Now suppose that $a\in Per(f)$ or $a$ is a super-attractor (the proof for $b$ is analogous).
Thus, there is $n>0$ such that $\lim_{\delta \downarrow 0} f^{n}(a+\delta)=a$ and $f^{j}(a)\notin[a,b)\ni c$, $\forall0<j<n$.
As $f^n$ is well defined, continuous and monotone on $(a,a+\varepsilon)$ for some $\varepsilon>0$ and as $f$ preserves orientation, we get $f^{n}(x)\in(a,b)$ for every $x>a$ sufficiently close to $a$ and that $f^{j}(x)\notin(a,b)$, $\forall  0<j<n$. Thus, there is some $I=(a,q)\in\cp_{J}$.
\cqd

\begin{Lemma}\label{LemmaHGFGH54}
Let $f:[0,1]\setminus\{c\}\to[0,1]$ be a $C^{2}$ non-flat contracting Lorenz map. If $J=(a,b)$ is a nice interval, then there are sequences $a_{n},b_{n}\in \overline{J}\cap Per(f)$ such that
\begin{enumerate}
\item $\lim_{n}a_{n}=a$ and $\lim_{n}b_{n}=b$;
\item $\co^{+}_{f}(a_{n})\cap(a_{n},b)=\emptyset$ and $\co^{+}_{f}(b_{n})\cap(a,b_{n})=\emptyset$.
\end{enumerate}
\end{Lemma}
\dem
We will show the existence of a sequence $a_{n}\in \overline{J}\cap Per(f)$ with $\lim_{n}a_{n}=a$ such that $\co^{+}_{f}(a_{n})\cap(a_{n},b)=\emptyset$.  Assume that $a\notin Per(f)$, otherwise take $a_{n}=a$. Let $I_{0}=(p_{0},q_{0})\in\cp_{J}$ such that $I_{0}\subset(a,c)$. By Lemma~\ref{Lemma8388881a}, as $a$ is not periodic we get $p_{0}\ne a$. Thus, there is some $I_{1}\in\cp_{J}$ with $I_{1}\subset (a,p_{0})$. In particular, $c\ne\partial I_{1}$. Again by Lemma~\ref{Lemma8388881a} we get $\cf_{J}(I_{1})=f^{n_{1}}(I_{1})=J$. Thus, there is a fixed point $a_{1}\in\overline{I_{1}}$ of $f^{n_{1}}|_{\overline{I_{1}}}$. As $n_{1}=R_{J}(I_{1})$ it follows that $f^{j}(a_{1})\notin(a,b)$ for every $0<j<n$ and so, $\{a_{1}\}=\co^{+}(a_{1})\cap(a,b)$. From this we get $\co^{+}(a_{1})\cap(a_{1},b)=\emptyset$. Again, writing $I_{1}=(p_{1},q_{1})$, it follows as before that $a\ne p_{1}$ and so there is some $I_{2}\in\cp_{J}$ such that $I_{2}\subset(a,p_{1})$. Proceeding as before, we get a periodic point $a_{2}\in\overline{I_{2}}$ satisfying the statement. Inductively, we get a sequence $a_{n}\searrow a$ of periodic points with $\co^{+}(a_{n})\cap(a_{n},b)=\emptyset$.
Similarly, one can get the sequence $b_{n}\nearrow b$.

\cqd

\begin{Lemma}
\label{noitenoite}
Let $f:[0,1]\setminus\{c\}\to[0,1]$ be a $C^{2}$ non-flat contracting Lorenz map.
If $Per(f)\cap (0,1)=\emptyset$, then either $f$ has an attracting periodic orbit (indeed, at least one of the fixed points is an attractor) or $\omega_{f}(x)\ni c$, $\forall\, x \in (0,1)$.
\end{Lemma}

\dem
Under these hypotheses, if $f$ has a periodic attractor, it has to be the point $0$, or $1$ or both. If none of these occur, $f$ does not have a periodic attractor. Suppose we can choose a point $x \in (0,1)$ such that $\omega_{f}(x)\not\ni c $. Let $(a,b)$ be the connected component of $[0,1]\setminus\overline{\co^+_f(x)}$ containing $c$. If $\exists n$ such that $f^n(a)\in(a,b)$, then, as $a\in\overline{\co^+_f(x)}$, $f^n(a)\in\overline{\co^+_f(x)}$, in contradiction with the fact that $(a,b) \subset [0,1]\setminus\overline{\co^+_f(x)}$. The same reasoning applies to point $b$, and then $(a,b) \subsetneq (0,1)$ is a nice interval and so $Per(f)\cap (a,b)\ne\emptyset$ (Lemma~\ref{LemmaHGFGH54}), which is a contradiction.
\cqd

\begin{Lemma}\label{Lemma545g55} 
Let $f:[0,1]\setminus\{c\}\to[0,1]$ be a non-flat $C^{2}$ contracting Lorenz map. 
If $f(x)>x$, $\forall x \in{(0,c)}$, $f(x)<x$, $\forall x \in{(c,1)}$ and $\lim_{x \uparrow c}f(x)>c > \lim_{x\downarrow c}f(x)$, then  $$\co^{+}_{f}(x)\cap(0,c)\ne\emptyset\ne\co^{+}_{f}(x)\cap(c,1), \hspace{0.2cm} \forall\,x\in (0,1)\setminus\co^{-}_f(c).$$
\end{Lemma}
\dem
Suppose, for instance, that there is $y\in(0,c)\setminus\co^{-}_f(c)$ such that $f^n(y)\in(0,c)$, $\forall\,n\ge0$. That is, $0<f^{n}(y)<c$ for all $n\ge0$.
As $f|_{(0,c)}$ is an increasing map, we get $f(0)=0<y<f(y)<f^{2}(y)<\cdots<f^{n}(y)\cdots<c$. This implies that $\lim_{n\to\infty}f^{n}(y)$ is a fixed point for $f$, contradicting our hypothesis. Thus, there is some $n>0$ such that $f^{n}(y)\in(c,1)$.
\cqd

\begin{Lemma}\label{AcumulacaoDePer}
If $f:[0,1]\setminus\{c\}\to[0,1]$ is a non-flat $C^{2}$ contracting Lorenz map without periodic attractors, then either $\exists\,\delta>0$ such that $c\in\omega_{f}(x)$, $\forall\,x\in(c-\delta,c+\delta)$ or $$\overline{Per(f)\cap(0,c)}\ni c\in\overline{(c,1)\cap Per(f)}.$$

\end{Lemma}

\dem
Suppose that $f$ does not have periodic attractors and suppose also that $\nexists\delta>0$ such that $c\in\omega_{f}(x)$, $\forall\,x\in(c-\delta,c+\delta)$. In this case, by Lemma~\ref{noitenoite}, $Per(f)\cap(0,1)\ne\emptyset$. As $f$ does not have periodic attractors, $f(x)>x$, $\forall x \in{(0,c)}$, $f(x)<x$, $\forall x \in{(c,1)}$ and $\lim_{x \uparrow c}f(x)>c > \lim_{x\downarrow c}f(x)$, and then $\co_{f}^{+}(x)\cap(0,c)\ne\emptyset\ne(c,1)\cap\co_{f}^{+}(x)$, $\forall\,x\in(0,1)\setminus\co^{-}_f(c)$, by Lemma \ref{Lemma545g55}. Thus, $Per(f)\cap(0,c)\ne\emptyset\ne(c,1)\cap Per(f)$.

Let $a=\sup Per(f)\cap(0,c)$ and $b=\inf Per(f)\cap(c,1)$. We know that $0<a\le c\le b<1$. If $a=b$ the proof is done. So suppose that $a\ne b$. We may assume that $0<a<c\le b<1$ (the other case is analogous).

We claim that $\co_{f}^{+}(a_{-})\cap(a,b)=\emptyset=(a,b)\cap\co_{f}^{+}(b_{+})$. Indeed, if there is a minimum $\ell\ge1$ such that $f^{\ell}(a_{-})\in(a,b)$, then $\emptyset\ne f^{\ell}((a-\varepsilon,a)\cap Per(f))\subset(a,b)$, contradicting the definition of $a$ and $b$. With the same reasoning we can show that $\co_{f}^{+}(b_{+})\cap(a,b)=\emptyset$.

Notice that $\exists n>0$ such that $f^{n}((a,c))\cap(a,c)\ne\emptyset$. Indeed, $(a,c)$ can not be  a wandering interval (Lemma~\ref{LemmaWI}) and as $f$ does not have periodic attractors, it follows from the homterval lemma (Lemma~\ref{homtervals}) that $f^{n}((a,c))\ni c$ for some $n\ge1$. Let $\ell$ be the smallest integer bigger than $0$ such that $c \in f^{\ell}((a,c))$. As $\co_{f}^{+}(a)\cap(a,c)=\emptyset$, we get $f^{\ell}((a,c))\supset(a,c)$. Thus, there is a periodic point $p\in[a,c)$ with period $\ell$. By the definition of $a$, it follows that $p=a$.

We claim that $f^{\ell}((a,c))\subset(a,b)$. If not, let $q_{0}\in f^{\ell}((a,c))\cap Per(f)\cap[b,1)$. Let $q=\min \co_{f}^{+}(q_{0})\cap(c,1)$ and $q'=(f^{\ell}|_{(a,c)})^{-1}(q)$. Clearly, $a<q'<c<q$ and $(q',q)$ is a nice interval. Thus, by Lemma~\ref{LemmaHGFGH54}, $Per(f)\cap(q',c)\ne\emptyset$ and this contradicts the definition of $a$.

Notice that $f^{\ell}((a,c))\ni c$, otherwise $f$ would have periodic attractors. As a consequence of this and of the claim above, $b>c$. 

As $b>c$, $(a,b)$ is a nice interval. We already know that $f^{\ell}(a)=a$. Moreover, by the definition of $b$ and Lemma~\ref{LemmaHGFGH54}, $b$ also must be a periodic point. So, let $r=\period(b)$. From the same reasoning of the claim above, we get $f^{r}((c,b))\subset((a,b))$.

Thus, the first return map to $[a,b]$ is conjugated to a contracting Lorenz map $g:[0,1]\setminus\{c_{g}\}\to[0,1]$. As $\nexists\delta>0$ such that $c\in\omega_{f}(x)$, $\forall\,x\in(c-\delta,c+\delta)$, it follows that $\exists\,x\in[0,1]$ such that $c_{g}\notin\omega_{g}(x)$. So, it follows from Lemma~\ref{noitenoite} that $Per(g)\cap(0,1)\ne\emptyset$. As a consequence, $Per(f)\cap(a,b)\ne\emptyset$. This contradicts the definition of $a$ and $b$, proving the lemma.

\cqd

\begin{Lemma}\label{omegaemc}
Let $f:[0,1]\setminus\{c\}\to[0,1]$ be a non-flat $C^{2}$ contracting Lorenz map without periodic attractors and such that $Per(f)\cap (0,1)=\emptyset$. If $x\in (0,1)$ is such that $c \in \omega_f(x)$, then $\overline{\co_f^+(x)\cap(0,c)}\ni c\in \overline{(c,1)\cap\co_f^+(x)}$.
\end{Lemma}
\begin{proof}
As $c \in \omega_f(x)$, $x$ is under the hypotheses of Lemma \ref{Lemma545g55}, then $\co_f^+(x)\cap(0,c)\ne\emptyset\ne(c,1)\cap\co_f^+(x)$. Also, $c \in \omega_f(x)$ implies that $c \in \overline{\co_f^+(x)\cap(0,c)}$ or $c\in \overline{(c,1)\cap\co_f^+(x)}$. Suppose one of these do not occur. For instance, suppose $c\not\in \overline{(c,1)\cap\co_f^+(x)}$. Then, defining $v=\inf \overline{(c,1)\cap\co_f^+(x)}$, we have that $J=(c,v)$ is such that $\nexists j$ such that $c\in f^j(J)$, for otherwise, either $J\subset f^j(J)$, that would imply the existence of a periodic repeller, what is in contradiction to the hypothesis, or $f^j(v) \in J$, and as $v \in \omega_f(x)$ and $\omega_f(x)$ is a positively invariant set, this is in contradiction with the definition of $J$.
So, as $\nexists j \in \NN$ such that $c \in f^j(J)$ and we are supposing there are no periodic attractors, Lemma \ref{homtervals} implies that $J$ is a wandering interval. But we know from Lemma \ref{LemmaWI} that $(c,v)$ cannot be a wandering interval, leading to an absurd. 

\end{proof}

\begin{Lemma}[Variational Principle]\label{LemmaVarPric}
Let $f:[0,1]\setminus\{c\}\to[0,1]$ be a non-flat $C^{2}$ contracting Lorenz map without periodic attractors. Suppose that $\nexists\delta>0$ such that $c\in\omega_{f}(x)$, $\forall\,x\in(c-\delta,c+\delta)$.
Given $\varepsilon>0$, there exists a unique periodic orbit minimizing the period of all periodic orbits intersecting $(c-\varepsilon,c)$. Similarly, there exists a unique periodic orbit minimizing the period of all periodic orbits intersecting $(c,c+\varepsilon)$.
\end{Lemma}
\dem
As $Per(f)\cap(c-\varepsilon,c)\ne\emptyset$ (Lemma~\ref{AcumulacaoDePer}), let
$$
n=\min\{\period(x)\,;x\in Per(f)\cap(c-\varepsilon,c)\}
$$
and suppose that there are $p_0,q_0\in Per_n(f)\cap(c-\varepsilon,c)$ such that $\co_f^+(p_0)\ne\co_f^+(q_0)$. Let $p=\max\{\co_f^+(p_0)\cap(c-\varepsilon,c)$ and $q=\max\{\co_f^+(q_0)\cap(c-\varepsilon,c)$. Thus, $\co_f^+(p)\cap(p,c)=\emptyset=\co_f^+(q)\cap(q,c)$. We may assume that $q<p$.

\begin{figure}[H]
\begin{center}\label{Fig47488}
\includegraphics[scale=.25]{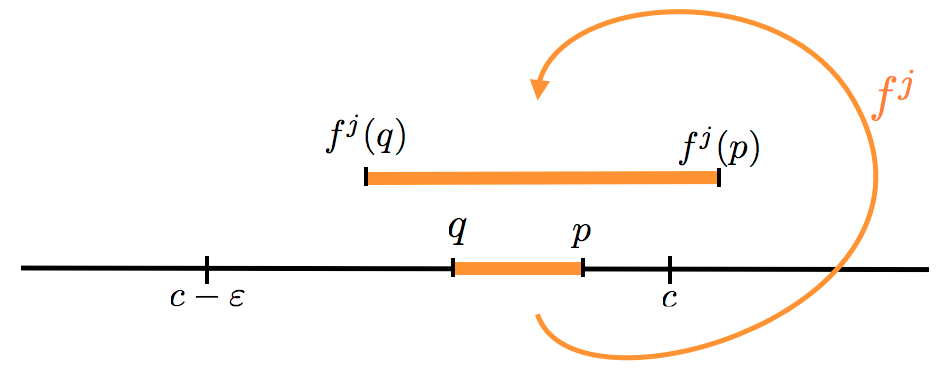}\\
\caption{}\label{Fig47488}
\end{center}
\end{figure}

Notice that $f^n$ can not be monotone on $(q,p)$. Indeed, otherwise, if $f^n$ is monotone on $(q,p)$, then $f^n([q,p])=[q,p]$. As  $f^n$ can not be the identity on $[q,p]$,  $f^n([q,p])=[q,p]$ would imply the existence of an attracting fixed point for $f^n$ on $[q,p]$. But this is impossible, as we are assuming that $f$ does not have a finite attractor.

As $f^n$ is not monotone on $(q,p)$, there is $0<j<n$ such that $f^j$ is monotone on $(q,p)$ and $c\in f^j((q,p))$. Thus, $f^j(q)<c<f^j(p)$. Moreover, $f^j(q)<q$ (because $\co_f^+(q)\cap(q,c)=\emptyset$ and $j<n$). Thus, $f^j((q,p))\supset(q,p)$ (see Figure~\ref{Fig47488}) and this implies in the existence of a periodic point $a\in[q,p]\subset(c-\varepsilon,c)$ with period $j<n$, contradicting the minimality of $n$.

The proof for the case $(c,c+\varepsilon)$ is analogous.

\cqd

\section{Renormalization and Cherry maps}
\label{renormecherry}

\begin{Definition}[Left and right renormalizations]
Let $f$ be a contracting Lorenz map, $J=(a,b)\in\cn$ and let $F:J^{*}\to J$ be the map of first return to $J$. We say that $f$ is renormalizable by the left side with respect to $J$ (or, for short, $J$-left-renormalizable) if $(a,c)\subset J^{*}$ (this means that $F|_{(a,c)}=f^{n}|_{(a,c)}$ for some $n\ge1$).
Analogously, we define $f$ to be renormalizable by the right side with respect to $J$ (or, for short, $J$-right-renormalizable) if $(c,b)\subset J^{*}$.
\end{Definition}

If the first return map to an interval $\overline{J}\ne[0,1]$, $F$, is conjugated to a Lorenz map,  $f$ is called {\em renormalizable} with respect to $J$. The renormalization of $f$ (with respect to $J$) is the map $g:[0,1]\setminus \{\frac{c-a}{b-a}\}\to[0,1]$ given by $$g(x)=A^{-1}\circ F \circ A(x)$$ where $A(x)=(b-a) x + a$.

Notice that $f$ is renormalizable with respect to $J$ if and only if $J\in\cn_{per}$ and $f$ is renormalizable by both sides (left and right) with respect to $J$. Moreover, using Corollary~\ref{Corollary8388881b}, it is easy to check the following result:

\begin{Lemma}\label{Lemma09090863}
Let $J=(a,b)\in\cn$. The following statements are equivalent:
\begin{enumerate}
\item $f$ is renormalizable with respect to $J$.
\item $(\,\overline{J}\,)\,^{*}=[a,c)\cup(c,b]$.
\item $c\in\partial I$, $\forall\,I\in (\overline J)^{*}$.
\item $a$ and $b$ are periodic points, $$f^{\period(a)}([a,c))\subset[a,b]\supset f^{\period(b)}((c,b]).$$
\end{enumerate}
\end{Lemma}

The interval involved in a (left/right) renormalization is called an interval of  (left/right) renormalization. A map $f$ is {\em non-renormalizable} if it does not admit any interval of renormalization.

In what follows, given a renormalization interval  $J $, we will refer to some concepts that were previously introduced. Namely, of its renormalization cycle $U_J$, the nice trapping region $K_{J}$ associated to J, and gaps of sets $\Lambda_{J}$ ($\Lambda_{J}$ also already defined, being the set of points whose orbits never reach an open set  $J \ni c$). These definitions were given before the statement of Theorem \ref{SOLENOIDETH} in Section \ref{mainresults}.

\begin{Lemma}\label{Corollary545g55}  Let $f:[0,1]\setminus\{c\}\to[0,1]$ be a $C^{2}$ contracting Lorenz map without periodic attractors.
For any given $J=(a,b)$ renormalization interval of $f$, we have that  $$\co^{+}_{f}(x)\cap(a,c)\ne\emptyset\ne\co^{+}_{f}(x)\cap(c,b)\hspace{0.5cm}\forall\,x\in J\setminus\co^{-}_f(c).$$
Therefore, the positive orbit $\co^{+}_{f}(x)$ of any $x\in J\setminus\co^{-}_f(c)$ intersects each connected component of the renormalization cycle $U_J$ (and also each connected component of the nice trapping region $K_{J}$).
\end{Lemma}

\dem
Let $\ell=\period(a)$ and $r=\period(b)$. As $f$ has no periodic attractors, it doesn't have any super-attractor, then $\lim_{x \uparrow c}f^{\ell}(x)>c > \lim_{x\downarrow c}f^{r}(x)$. 
If there is $x\in(a,c)$ such that $f^{\ell}(x)\le x$, then $f^{\ell}|_{[a,x]}$ will have an attracting fixed point, as $f^{\ell}|_{[a,x]}$ is not the identity, but this contradicts the hypothesis. The same reasoning can be done for $f^{r}|_{(c,b)}$, and therefore applying Lemma \ref{Lemma545g55} to the renormalization of $f$ with respect to $J$, we conclude the proof. 
\cqd

We say that two open intervals $I_{0}$ and $I_{1}$ are linked if $\partial I_{0}\cap I_{1}\ne\emptyset\ne I_{0}\cap\partial I_{1}$.

\begin{Lemma}\label{renormalinks} Let $f:[0,1]\setminus\{c\}\to[0,1]$ be a $C^{2}$ non-flat contracting Lorenz map without periodic attractors. Then, two renormalization intervals of $f$ can never be linked. Moreover, if $J_{0}$ and $J_{1}$ are two renormalization intervals and $J_{0}\ne J_{1}$, then  either $\overline{J_{0}}\subset J_{1}$ or $\overline{J_{1}}\subset J_{0}$. In particular, $\partial J_{0}\cap\partial J_{1}=\emptyset$.

\end{Lemma}
\dem

Write $J_0=(a_0,b_0)$ and $J_1=(a_1,b_1)$. First note that $J_{0}$ and $J_{1}$ can not be linked. Indeed, if they were linked, we would either have $a_{0}<a_{1}<c<b_{0}<b_{1}$ or $a_{1}<a_{0}<c<b_{1}<b_{0}$. We may suppose that $a_{0}<a_{1}<c<b_{0}<b_{1}$. In this case, $a_{1}\in J_{0}$ and by Lemma~\ref{Corollary545g55}.
$\emptyset\ne\co_{f}^{+}(a_{1})\cap(c,b_{0})\subset\co_{f}^{+}(a_{1})\cap(a_{1},b_{1})=\co_{f}^{+}(a_{1})\cap J_{1}$ contradicting the fact that $J_{1}$ is a nice interval.

As $J_{0}\cap J_{1}\ne\emptyset$ (because both contains the critical point) and as $J_{0}$ and $J_{1}$ are not linked, it follows that either $J_{0}\supset J_{1}$ or $J_{0}\subset J_{1}$. We may suppose that $J_{0}\supset J_{1}$. In this case, as $J_{0}\ne J_{1}$ we have three possibilities: either $a_{0}< a_{1}<c<b_{1}= b_{0}$ or $a_{0}= a_{1}<c<b_{1}< b_{0}$ or $J_{0}\supset\overline{J_{1}}$. If $a_{0}< a_{1}<c<b_{1}= b_{0}$, we can use again Lemma~\ref{Corollary545g55} to get  $\co_{f}^{+}(a_{1})\cap J_{1}\ne\emptyset$. On the other hand, if $a_{0}= a_{1}<c<b_{1}< b_{0}$, the same  Lemma~\ref{Corollary545g55} implies that $\co_{f}^{+}(b_{1})\cap J_{1}\ne\emptyset$. In both cases we get a contradiction to the fact that $J_{1}$ is a nice interval. Thus, the remaining possibility is the only valid one.

\cqd

A periodic attractor $\Lambda$ is called {\em essential} if its local basin contains $c^{-}$ or $c^{+}$. Precisely, if $\exists\,p\in\Lambda$ such that $(p,c)$ or $(c,p)$ is contained in $\beta(\Lambda)=\{x\,;\,\omega_{f}(x)\subset\Lambda\}$ (the basin of $\Lambda$). If a periodic attractor in not essential, it is called {\em inessential}. Notice that if $f$ is $C^{3}$ and has negative Schwarzian derivative, then, by Singer's Theorem, $f$ does not admit inessential periodic attractors.

\begin{Proposition}
\label{Lemma1110863}Suppose that $f:[0,1]\setminus\{c\}\to[0,1]$ is a $C^{2}$ non-flat contracting Lorenz map that does not admit inessential periodic attractors. 
If $J_{n}$ is an infinite sequence of renormalization intervals with $J_{n}\supsetneqq J_{n+1}$, then $\bigcap_{n}J_{n}=\{c\}$.
\end{Proposition}
\dem

Let $J=\bigcap_{n}J_{n}$. Write $(a,b)=\interior J$. Suppose for example that $a\ne c$ (the case $b\ne c$ is analogous). Given $x\in(a,b)$, let $R(x)=\min\{j>0\,;\,f^{j}(x)\in(a,b)\}$. As $J_n=(a_n,b_n)$ are renormalization intervals, then $(a_n,c)$ only returns to  $J_n$ at $\period(a_n)$ (and $(c,b_n)$ at the period of $b_n$), that is, the first return is at the time $\period(a_n)$. So, as $R(x)\ge\min\{\period(a_n),\period(b_{n})\}\to\infty$. Thus, $R(x)=\infty$, $\forall\,x\in(a,b)$. 
As $f^{j}((a,c))\cap(a,b)=\emptyset$, $\forall\,j>0$ (because $R\equiv\infty$), then $f^{j}|_{(a,c)}$ is a homeorphism $\forall\,j$. By Lemma~\ref{LemmaWI}, $(a,c)$ is not a wandering interval. As $\co_f^-(Per(f))$ does not contain intervals, it follows from Lemma~\ref{homtervals} that there is a periodic attractor $\Lambda$ with $(a,c)\cap\beta(\Lambda)\ne\emptyset$. As $f$ does not have inessential  periodic attractors, there is some $q\in\Lambda$ such that $(q,c)$ or $(c,q)\subset\beta(\Lambda)$. 
As $q$ is periodic, $q\notin[a,b]$. Thus, $q<a_{n}<c$ for some $n$ or $c<b_{n}<q$. In any case, we get a contradiction for nor $a_n$ neither $b_n$ can be in the basin of a periodic attractor.

\cqd

\begin{Corollary}\label{Corolary989982}
Suppose that $f:[0,1]\setminus\{c\}\to[0,1]$ is a $C^{2}$ non-flat contracting Lorenz map that does not admit inessential periodic attractors. If there exists $p\in(0,1)$ such that $\alpha_{f}(p)\ni c\notin\overline{\co_{f}^{+}(p)}$, then $f$ is not an infinitely renormalizable map.
\end{Corollary}

\dem

Suppose that $T_{n}$ is a sequence of two by two distinct renormalizable intervals. By Proposition~\ref{Lemma1110863}, $\bigcap_{n}T_{n}=\{c\}$.  
For each $n\in\NN$, let $0<r_n,\ell_n\in\NN$ be such that $f^{\ell_n}(T_n\cap(0,c))\subset T_n$ and $f^{r_n}(T_n\cap(c,1))\subset T_n$ and let $$U_{n}=T_n\cup\bigg(\bigcup_{j=1}^{\ell_n-1}f^j((T_n\cap(0,c))\bigg)\cup\bigg(\bigcup_{j=1}^{r_n-1}f^j((T_n\cap(c,1))\bigg).$$

If $p\in U_{n}$, $\forall\,n \in \NN$, then $c\in\omega_{f}(p)$, contradicting our hypothesis. Thus, one can find some $n\ge0$ such that $p\notin U_{n}$. But this is not possible, because $c\in\alpha_{f}(p)$ and so, $\co_{f}^{-}(p)\cap T_{n}\ne\emptyset$.

\cqd

Now we have enough information on maps that are infinitely many times renormalizable in order to prove Theorem \ref{SOLENOIDETH}.

\dem[Proof of Theorem \ref{SOLENOIDETH}]

Write $\cR =\{J_n\}_{n\in\NN}$, with $J_1\supsetneqq J_2\supsetneqq J_3\supsetneqq\cdots$.
Notice that $J_n\supset\overline{J_{n+1}}$, $\forall\,n \in \NN$ and also
\begin{equation}
K_{J_n}=\interior(K_{J_n})\supset\overline{K_{J_{n+1}}},\;\;\forall\,n\in\NN.
\end{equation}
Thus, $$\Delta:=\bigcap_{n\in\NN}\overline{K_{J_n}}=\bigcap_{n\in\NN}K_{J_n}.$$

As each $K_n$ is a trapping region ($f(K_n)\subset K_n$), it is easy to see that $\omega_f(x)\subset\Delta$, whenever $c\in\omega_f(x)$. Indeed, if  $c\in\omega_f(x)$, then $\co^+_f(x)\cap J_n\ne\emptyset$ for every $n\in\NN$, because $\{c\}=\bigcap_{n}J_n$ (see Proposition~\ref{Lemma1110863}). Thus, $\omega_f(x)\subset\overline{K_n}$, $\forall\,n \in \NN$.

Let $\ck_n$ be the collection of connected components of $K_{J_n}$ and $\ck_n(y)$ be the element of $\ck_n$ containing $y$ (see Figure~\ref{KnBasica.png}), for any given $y\in\Delta$.
Let $\Lambda$ be the (closed) set of points $y\in\Delta$ such that there is a sequence $\Delta\ni y_n\to y$ and $\NN\ni k_n\to\infty$ with $\lim_n\diameter(\ck_{k_n}(y_n))=0$. 
Given any $x\in[0,1]$ with $c\in\omega_f(x)$, we have $\co^+_f(x)\cap J_n\ne\emptyset$  $\forall\,n\in\NN$ and,  by Proposition~\ref{Lemma1110863} and Lemma~\ref{Corollary545g55}, $\co^+_f(x)$ intersects every element of $\ck_n$, $\forall\,n \in \NN$. As a consequence, any point $y \in \Lambda$ is accumulated by points of $\co^+_f(x)$ for any $x\in[0,1]$ with $c\in\omega_f(x)$.  That is,
\begin{equation}\label{eq6783453}
\Delta\supset\omega_f(x)\supset\Lambda\text{\, for every $x$ such that }c\in\omega_f(x).
\end{equation}

\begin{figure}
\begin{center}\label{KnBasica.png}
\includegraphics[scale=.2]{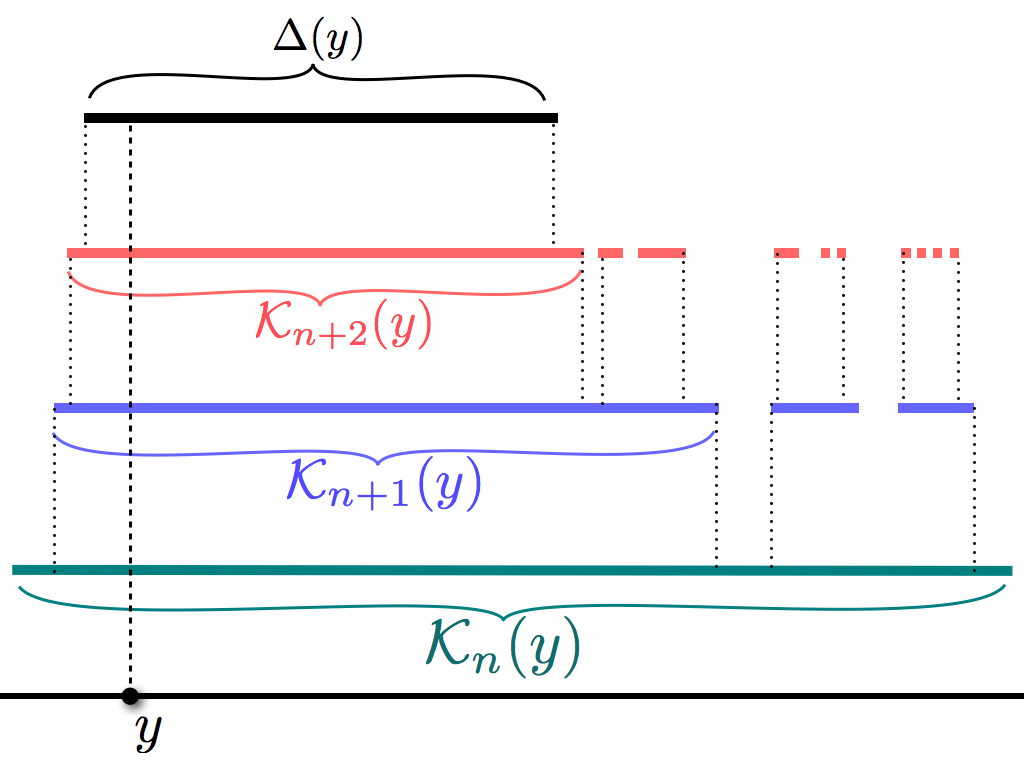}\\
\caption{}\label{KnBasica.png}
\end{center}
\end{figure}

\begin{claim} Define $\Delta(y)$ as the connected component of $\Delta$ containing $y$.
If $\interior(\Delta(y))\ne\emptyset$, $y\in\Delta$, then $\interior(\Delta(y))$ is a wandering interval.
\end{claim}
\dem[Proof of the claim]

Suppose that $\interior(\Delta(y))\ne\emptyset$ and that $\exists s$ such that $c \in f^s(\Delta(y))$. Then, $\forall n$, $f^s(\Delta(y))\cap J_n\ne\emptyset$. But if $f^s(\Delta(y))\cap J_n\ne\emptyset$, then $f^s(\ck_n(y))\cap J_n\ne\emptyset$ and so, $f^s(\Delta(y))\subset f^s(\ck_n(y))\subset J_n$, $\forall n$.
Thus, if $c\in f^s(\Delta(y))$, then we have $f^s(\Delta(y))\subset \bigcap_nJ_n=\{c\}$ (Proposition~\ref{Lemma1110863}), a contradiction. This implies that $c\notin f^s(\Delta(y))$, $\forall s\in\NN$. From Lemma~\ref{homtervals}, we get that $\interior(\Delta(y))$ is a wandering interval.
\cqd

Now consider $y\in\Delta\setminus\Lambda$. We will show that if $c \in \omega_f(x)$, then $y \not\in \omega_f(x)$. 

Under the assumption of $y\in\Delta\setminus\Lambda$, there is some $\varepsilon>0$ such that $B_{\varepsilon}(y)\cap\Delta=B_{\varepsilon}(y)\cap\Delta(y)$.
Notice that $\Delta(y)\ne\{y\}$, otherwise $\lim_n\diameter(\ck_n(y))=0$ and $y\in\Lambda$.
So, $\interior(\Delta(y))\ne\emptyset$ and so, by the claim above, $\interior(\Delta(y))$ is a wandering interval. This implies that $\omega_f(x) \cap \interior(\Delta(y))=\emptyset$, $\forall x$. So, if $y\in\interior(\Delta(y))$ we have that 
$y \not\in \omega_f(x)$. 

Let's then consider $y\notin\interior(\Delta(y))$. Reducing $\epsilon$ if necessary, $B_{\varepsilon}(y)\cap\Delta\cap\Omega(f)=B_{\varepsilon}(y)\cap\Delta(y)\cap\Omega(f)\subset\{y\}$. Suppose that $y\in\omega_f(x)$ for some $x$ such that $c\in\omega_f(x)$.
In this case, as $\Delta\supset\omega_f(x)$, we conclude that $y$ is an isolated point of $\omega_f(x)$: indeed, as $\omega_f(x)\subset\Delta\cap\Omega(f)$, we have $y \in \omega_f(x)\cap B_{\varepsilon}(y)=\omega_f(x)\cap B_{\varepsilon}(y)\cap\Delta\cap\Omega(f)=\omega_f(x)\cap B_{\varepsilon}(y)\cap\Delta(y)\cap\Omega(f)\subset\{y\}$, then this set is $\{y\}$.

\begin{figure}
\begin{center}\label{Kn.pdf}
\includegraphics[scale=.3]{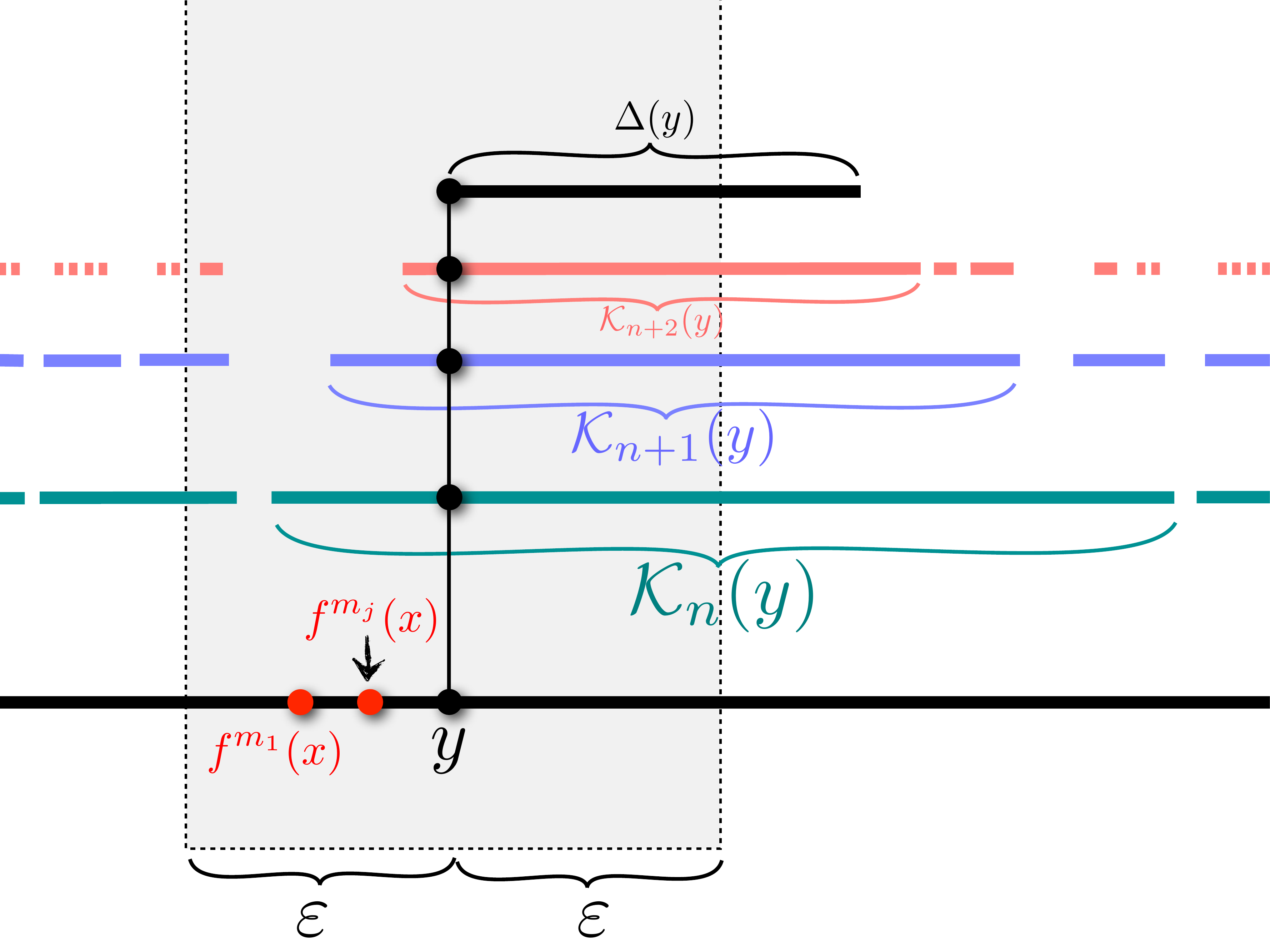}\\
\caption{}\label{Kn.pdf}
\end{center}
\end{figure}

Since $y\notin\interior(\Delta(y))$, we may suppose that $\Delta(y)=[y,b]$ (the case $\Delta(y)=[a,y]$ is analogous). Taking $\varepsilon>0$ small enough, we can assume that $y+\varepsilon<b$. Let $n\ge1$ be such that $y-\varepsilon<k_{n,0}(y)<y$, where $(k_{n,0},k_{n,1}):=K_n(y)$. Let $m_j\in\NN$ be such that $k_{n,0}<f^{m_1}(x)<f^{m_2}(x)<\cdots<f^{m_j}(x)\nearrow\,y$ and $\co_f^+(x)\cap(k_{n,0},y)=\{f^{m_1}(x),f^{m_2}(x),f^{m_3}(x),\cdots\}$ (See Figure~\ref{Kn.pdf}). 

Choose $j_0$ big enough so that $m_j>m_1$, $\forall\,j\ge j_0$. Given $j\ge j_0$, let $I_j=(t_j,f^{m_1}(x))$ be an interval contained in $(k_{n,0},f^{m_1}(x))$, maximal such that $f^{m_j-m_1}|_{I_j}$ is a homeomorphism. If $k_{n,0}<t_j$, there is some $1\le s<m_j-m_1$ s.t. $f^s((t_j,f^{m_1}(x)))=(c,f^{m_1+s}(x))$. As $f$ is infinitely renormalizable, Lemma \ref{Corollary545g55} says that the orbit of $x$ accumulates on $c$ by both sides, then $\#\co_f^+(x)\cap (c,f^{m_1+s}(x))=\infty$. As $K_{J_n}$ is positively invariant and 
$f^{m_j-m_1}(I_j)\cap K_n(y)\ne\emptyset$, we get that $f^{m_j-m_1}(I_j)\subset K_n(y)$. So, 
$f^{m_j-m_1}(I_j)\subset (k_{n,0},f^{m_j}(x))\subset (y-\varepsilon,f^{m_j}(x))$. Thus, $\#\co_f^+(x)\cap(y-\varepsilon,f^{m_j}(x))=\infty$, and this is an absurd, as $y$ was taken was the only non-wandering point in this neighborhood. 

Thus, $t_j=k_{n,0}$  and so, $I_j=(k_{n,0},f^{m_1}(x))\,\,\forall\,j\ge j_0.$

As a consequence, $f^j|_{(k_{_{n,0}},f^{m_1}(x))}$ is a homeomorphism $\forall\,j\in\NN$  because $f^{m_j-m_1}|_{(k_{_{n,0}},f^{m_1})}=f^{m_j-m_1}|_{I_j}$ is a homeomorphism, $\forall\,j\ge j_0$. But this contradicts the homterval lemma (Lemma~\ref{homtervals}), as $(k_{_{n,0}},f^{m_1}(x))$ cannot be a wandering interval ($k_{n,0}$ is pre-periodic, as $\partial K_{J_n}\subset \co^{-}_{f}(\partial J_n)$ ) and as $f$ does not have periodic attractors.

For short, if $c\in\omega_f(x)$, then $y\notin\omega_f(x)$ for all $y\in\Delta\setminus\Lambda$. So, by (\ref{eq6783453}), $\omega_f(x)=\Lambda$ when $c\in\omega_f(x)$. Finally, as $\Lambda\subset\bigcap_{J\in\cR }K_J$ and $c\in\omega_f(x)$ for every $x\in\bigcap_{J\in\cR }K_J$, then $\omega_f(x)=\Lambda$, $\forall\,x\in\Lambda$. That is, $\Lambda$ is minimal and so we conclude the proof.

\cqd

\begin{Remark}\label{Remark12327890}
Let $x,y\in[0,1]\setminus\{c\}$, $\delta>0$ and $j\in\NN$. If $f^{j}|_{(y-\delta,y+\delta)}$ is an homeomorphim, then $y\in\alpha_{f}(x)$ $\iff$ $f^{j}(y)\in\alpha_{f}(x)$.
\end{Remark}

\begin{Lemma}
\label{Lemma01928373} 
Let $f:[0,1]\setminus\{c\}\to[0,1]$ be a $C^{2}$ non-flat contracting Lorenz map without periodic attractor. If $c\in\alpha_{f}(p)$ for some $p\ne c$, then $\overline{\co^{-}_{f}(p)\cap[0,c)}\ni c \in\overline{(c,1]\cap\co^{-}_{f}(p)}$.
\end{Lemma}
\dem
Suppose that $c\in\alpha_{f}(p)$, $p\in[0,c)\cup(c,1]$.  We may suppose that $\co^{-}_{f}(p)\cap(-\delta,c)\ne\emptyset$, $\forall\delta>0$ and $\co^{-}_{f}(p)\cap(c,\delta_{0})=\emptyset$ for some $\delta_{0}>0$, the symmetrical case being analogous. As $\alpha_{f}(p)$ is compact, there is some $q>0$ such that $(c,q)$ is a connected component of $[0,1]\setminus\alpha_{f}(p)$.
\begin{claim}
$f^{j}\big((c,q)\big)\cap(c,q)=\emptyset$, $\forall\,j>0$.
\end{claim}
\dem[Proof of the Claim] Suppose there is a smallest $\ell>0$ such that $f^{\ell}\big((c,q)\big)\cap(c,q)\ne\emptyset$. In this case $f^{\ell}|_{(c,q)}$ is a homeomorphism.
If $f^{\ell}\big((c,q)\big)\subset(c,q)$, then $f$ admits a periodic attractor or a super-attractor, contradicting our hypothesis. Thus, there is some $x\in\{c,q\}\cap f^{\ell}\big((c,q)\big)$. As both $c$ and $q$ are accumulated by pre-images of $p$, it follows that $x$ is also accumulated by pre-images of $p$. So, $\alpha_{f}(p)\cap(c,q)\ne\emptyset$  (Remark~\ref{Remark12327890}), contradicting that $(c,q)$ is contained in the complement of $\alpha_{f}(p)$. 
(end of the proof of the Claim)\cqd

It follows from the Claim that $f^{j}|_{(c,q)}$ is a homeomorphism for every $j>0$. Moreover, $(c,q)$ is a wandering interval. Indeed, if $f^{j}\big((c,q)\big)\cap f^{k}\big((c,q)\big)\ne\emptyset$, with $j<k$, then $f^{j}\big((c,q)\big)\not\supset f^{k}\big((c,q)\big)$, since $f^{j}\big((c,q)\big)\supset f^{k}\big((c,q)\big)$ implies the existence of a periodic attractor or a super-attractor, contradicting again our hypothesis. Thus, there is $x\in\{f^{j}(c),f^{j}(q)\}$ belonging to $f^{k}\big((c,q)\big)$. As $f^{j}(c)$ and $f^{j}(q)\in\alpha_{f}(p)$ we get $\big(f^{j}|_{(c,q)}\big)^{-1}(x)\in\alpha_{f}(p)\cap(c,q)$ (Remark~\ref{Remark12327890}), contradicting again that $(c,q)$ is contained in the complement of $\alpha_{f}(p)$.

As $(c,q)$ being a wandering interval is a contradiction to Lemma~\ref{LemmaWI}, we have to conclude that $\co^{-}_{f}(p)\cap(c,\delta)\ne\emptyset$, $\forall\delta>0$.
\cqd

\begin{Lemma}\label{Lemma549164}
Let $f:[0,1]\setminus\{c\}\to[0,1]$ be a contracting Lorenz map without periodic attractors. Let $p\in(0,1)$ be such that $c\notin\overline{\co_{f}^{+}(p)}$ and let $(p_{1},p_{2})$ be the connected component of $(0,1)\setminus\overline{\co_{f}^{+}(p)}$ containing the critical point $c$. Given $y\in\co_{f}^{-}(p)$ and $\varepsilon>0$, we have $$\bigcup_{j\ge0}f^{j}(y,y+\varepsilon)\supset (p_{1},c)\,\,\text{ and }\,\,\bigcup_{j\ge0}f^{j}(y-\varepsilon,y)\supset (c,p_{2}).$$\end{Lemma}

\dem

For any given $\delta>0$, Lemma~\ref{Remark98671oxe} says that, $\forall \varepsilon >0$,  $\exists j_{1},j_{2}\ge0$ such that $f^{j_{1}}((p-\varepsilon,p))\cap(c-\delta,c+\delta)\ne\emptyset\ne f^{j_{2}}((p,p+\varepsilon))\cap (c-\delta,c+\delta)$. Take $j_{1},j_{2}$ minima with such property such that $f^{j_{1}}|_{(p-\varepsilon,p)}$ and $f^{j_{2}}|_{(p,p+\varepsilon)}$ are homeomorphisms. Notice that $f^{j_{2}}|_{(p,p+\varepsilon)}\supset(p_{1},c-\delta)$ and $f^{j_{1}}((p-\varepsilon,p))\supset(c+\delta,p_{2})$, as $\co_{f}^{+}(p)\cap(p_{1},p_{2})=\emptyset$ and $f^{j_{1}}|_{(p-\varepsilon,p)}$ and $f^{j_{2}}|_{(p,p+\varepsilon)}$ preserve orientation.

As a consequence, $$\bigcup_{j\ge0}f^{j}\big((p,p+\varepsilon)\big)\supset \bigcup_{\delta>0}(p_{1},c-\delta)=(p_{1},c)$$ and $$\bigcup_{j\ge0}f^{j}\big((p-\varepsilon,p)\big)\supset\bigcup_{\delta>0}(c+\delta,p_{2})=(c,p_{2}).$$

Suppose that $y\in f^{-s}(p)$ for some $s\ge1$. There is $r>0$ such that $f^{s}|_{(y,y+r)}$ and $f^{s}|_{(y-r,y)}$
are homeomorphisms. As $f^{s}|_{(y,y+r)}$ is a homeomorphism, $f^{s}((y,y+r))=(p,p+\varepsilon)$ with $\varepsilon=f^{s}(y+r)-p$. Thus, $$\bigcup_{j\ge0}f^{j}\big((y,y+r)\big)\supset\bigcup_{j\ge0}f^{j}\big(f^{s}((y,y+r))\big)=\bigcup_{j\ge0}f^{j}\big((p,p+\varepsilon)\big)\supset(p_{1},c).$$
\cqd

\begin{Lemma}\label{jotaxrenormaliza}
Let $f:[0,1]\setminus\{c\}\to[0,1]$ be a contracting Lorenz map. Write $v_{1}=f(c_-)$ and $v_{0}=f(c_+)$.
Given any $x$, $v_{0} < x < v_{1}$, let $J_{x}=(x_{1},x_{2})$ be the connected component of $[0,1]\setminus\alpha_{f}(x)$ that contains the critical point $c$. If $J_{x}\ne\emptyset$, then $J_{x}$ is a renormalization interval and $\partial J_{x}\subset\alpha_{f}(x)$.
\end{Lemma}
\dem
Firstly observe that $\alpha_{f}(x)\supset\{0,1\}$ because, as $x \in (v_{0},v_{1})$, $0=\lim_{n\to\infty}(f|_{[0,c)})^{-n}(x)$ and $1=\lim_{n\to\infty}(f|_{(c,1]})^{-n}(x)$. Thus, $J_{x}$ is an open interval. Moreover, $\partial J_{x}\subset\alpha_{f}(x)$.

We claim that $J_{x}$ is a nice interval. Otherwise, consider $n$ the smallest integer $n>0$ such that $f^{n}(\partial J_{x})\cap J_{x}\ne\emptyset$. Let $i\in\{1,2\}$ be so that $f^{n}(x_{i})\in J_{x}$. As $f^{j}(x_{i})\notin J_{x}$, $\forall0\le j<n$, there is $\varepsilon>0$ such that $f^{n}|_{(x_{i}-\varepsilon,x_{i}+\varepsilon)}$ is a homeomorphism. From Remark~\ref{Remark12327890} it follows that $f^{n}(x_{i})\in\alpha_{f}(x)$, contradicting $\alpha_{f}(x)\cap J_{x}=\emptyset$. Thus, $J_{x}\in\cn$.

Now let us check that $J_{x}$ is a renormalization interval. Suppose it is not the case, it follows from Lemma~\ref{Lemma09090863} that one can find a connected component $I=(t_{1},t_{2})$ of the domain of the first return map to $J_{x}$ such that $c\notin\partial I$. By Lemma~\ref{Lemma8388881a}, $f^{k}(I)=\cf_{J_{x}}(I)=J_{x}$, where $k=R_{J_{x}}(I)$. Notice that $t_{1}$ or $t_{2}\in(x_{1},x_{2})$. Suppose that $t_{1}\in(x_{1},x_{2})$ (the case $t_{2}\in(x_{1},x_{2})$ is similar). As $c\notin\partial I$ (and $f^{j}(t_{1})\notin J_{x}$, $\forall0<j<k$), there is some small $\delta>0$ such that $f^{k}|_{(t_{1}-\delta,t_{1}+\delta)}$ is a homeomorphism. As $f^{k}(t_{1})=x_{1}\in\alpha_{f}(x)$, it follows from Remark~\ref{Remark12327890} that $t_{1}\in\alpha_{f}(x)$. But this is impossible as $\alpha_{f}(x)\cap J_{x}=\emptyset$.
\cqd

\begin{Corollary}\label{Cor111} Let $f:[0,1]\setminus\{c\}\to[0,1]$ be a $C^{2}$ non-flat contracting Lorenz map without periodic attractors. If $p\in Per(f)$, then either
$\overline{\co^{-}_{f}(p)\cap(0,c)}\ni c \in\overline{(c,1)\cap\co^{-}_{f}(p)}$ or the connected  component of $[0,1]\setminus\alpha_{f}(p)$, $J_{p}$, is non-empty and it is a renormalization interval.\end{Corollary}

\begin{Notation}[$\cl_{Per}$, $\cl_{Sol}$ and $\cl_{Che}$]
\label{persolche}
Let $\cl_{Per}$ denote the collection of contracting Lorenz maps having periodic attractors. The set of all $\infty$-renormalizable contracting Lorenz maps will be denoted by $\cl_{Sol}$. Let $\cl_{Che}$ be the set of all contracting Lorenz maps that are Cherry maps.
\end{Notation}

Recall that $f$ is a {\em Cherry map} if it does not have a periodic or super-attractor and there is $\delta>0$  such that $c\in\omega_{f}(x_{\pm})$ for every $x\in(c-\delta,c+\delta)$.

\begin{Lemma}
\label{vizinhanca}
If $f:[0,1]\setminus\{c\}\to[0,1]$ is a $C^{2}$ non-flat contracting Lorenz map and $f \notin \cl_{Per}\cup\cl_{Sol} \cup \cl_{Che}$, then $c\in\alpha_{f}(p)$ for some $p\in Per(f)$.
\end{Lemma}

\dem
If $f$ is not renormalizable let $I=(0,1)$, otherwise let $I=(a,b)$ be the smallest renormalization interval of $f$ (we are assuming that $f\notin \cl_{Per}\cup\cl_{Sol} \cup \cl_{Che}$).
By lemma \ref{noitenoite} we can pick a point $p \in (a,b)$ that is periodic. So, we have that $ p \in \alpha_{f}(p)$. As a consequence, it follows from Corollary~\ref{Cor111} that $\overline{\co_{f}^{-}(p)\cap(0,c)}\ni c \in\overline{(c,1)\cap \co_{f}^{-}(p)}$. Indeed, if the pre-orbit of $p$ is not accumulating on $c$ by both sides, then $J_{p}\ne\emptyset$ is a renormalization interval. In this case, as $p\in \alpha_{f}(p)$, we get $J_{p}\subsetneqq(a,b)$. This is an absurd,  as $(a,b)$ is the smallest renormalization interval.

\cqd

\begin{Proposition}[Long branches lemma]\label{Proposition008345678}Let $f:[0,1]\setminus\{c\}\to[0,1]$ be a $C^{2}$ non-flat contracting Lorenz map.
Suppose that $f$ does not admit a periodic attractor. If $\alpha_{f}(p)\ni c \notin\omega_{f}(p)$ for some $p\ne c$, then there exists $\varepsilon>0$ such that $\overline{ \co^{-}_{f}(x)\cap(0,c)} \ni c \in \overline{ \co^{-}_{f}(x)\cap(c,1)}$ for every $0<|x-c|<\varepsilon$. Furthermore, $f$ is not $\infty$-renormalizable, $f$ is not a Cherry map and $Per(f)$ $\cap$ $(c-\delta,c)$ $\ne$ $\emptyset$ $\ne$ $Per(f)$ $\cap$ $(c,c+\delta)$, $\forall\delta>0$.
\end{Proposition}
\dem
Suppose by contradiction that the main statement is not true. That is, $c\in\overline{W}$, where $$W=\{x\,;\,c\notin  \overline{ \co^{-}_{f}(x)\cap(0,c)} \text{ or }c\notin \overline{ \co^{-}_{f}(x)\cap(c,1)}\}.$$

By Lemma~\ref{Lemma01928373}, if $\co^{-}_{f}(x)$ accumulates on one side of $c$, then  $\co^{-}_{f}(x)$ will accumulate on $c$ by both sides. Then, $W=\{x; c \notin \alpha_f(x)\}$.

Let $(p_{1},p_{2})$ be the connected component of $[0,1]\setminus\overline{\co_{f}^{+}(p)}$ that contains $c$.  
Choose a sequence $\co_{f}^{-}(p)\ni y_{n}\to c$. As $f$ does not have a periodic attractor, taking a subsequence if necessary, we get by Lemma~\ref{Lemma549164} that 
\begin{equation}\label{EQ2345654a}
\bigcup_{j\ge0}f^{j}\big((y_{n},y_{n}+\varepsilon)\big)\supset (p_{1},c)\,\,\forall\varepsilon>0,\forall\,n>0
\end{equation} and that
\begin{equation}\label{EQ2345654b}
\bigcup_{j\ge0}f^{j}\big((y_{n}-\varepsilon,y_{n})\big)\supset (c,p_{2})\,\,\forall\varepsilon>0,\forall\,n>0.
\end{equation}

As $c$ is accumulated by $W$, say by the left side (the other case is similar), choose some $q\in(p_{1},c)\cap W$. It follows from (\ref{EQ2345654a}) that $\bigcup_{j\ge0}f^{j}\big((y_{n},c)\big)\supset (p_{1},c)\ni q$, $\forall\,n>0$ (we are taking $\varepsilon=|y_{n}-c|$ in (\ref{EQ2345654a})). Thus, there is a sequence $y_{n}<q_{n}<c$ and $i_{n}\to\infty$ such that $f^{i_{n}}(q_{n})=q$, $\forall\,n \in \NN$. This implies that $c\in\alpha_{f}(q)$. But this is an absurd because $q\in W$.

Therefore, we can not have $c\in\overline{W}$ and this proves the main part of the Proposition. By Corollary~\ref{Corolary989982}, $f$ cannot be $\infty$-renormalizable. As $\omega_{f}(y)=\omega_{f}(p)\not\ni c$ for all $y\in\co^{-}_{f}(p)$, it follows that $f$ cannot be a Cherry map. Finally, let us show that $Per(f)\cap(c-\delta,c)\ne\emptyset\ne Per(f)\cap(c,c+\delta)$, $\forall\delta>0$. For this, let $n\ge \period(p)$ and $J_{n}$ be the connected component of $(0,1)\setminus\bigcup_{j=0}^{n-1}f^{-j}(p)$ containing the critical point $0$. It is easy to see that $J_{n}$ is a nice interval, $\forall\,n \in \NN$. Also, as $\alpha_{f}(p)\ni c$, $\forall \delta>0$, $\exists n$ such that $\partial J_n \subset B_\delta(c)$. As it follows from Lemma~\ref{LemmaHGFGH54} that  $Per(f)\cap\overline{J_{n}}\cap(-\infty,c)\ne\emptyset\ne(c,+\infty)\cap\overline{J_{n}}\cap Per(f)$, $\forall\,n \in \NN$, we conclude the proof.

\cqd

Observe that it is also true that $f$ being a Cherry map implies that $Per(f) \cap (u,v)=\emptyset$,  $(u,v)$ being the last interval of renormalization.

\section{The structure of the Topological Attractors}
\label{TheSofTA}

We now study the topological attractors for the contracting Lorenz maps. The main result is Theorem \ref{cicloint}, from which we obtain (Section~\ref{ProofABCD}) the main theorems: Theorem \ref{baciastopologicas}, \ref{teoalfalim} and \ref{atratortopologico}.

In this Section, $f$ will be a $C^{2}$ non-flat contracting Lorenz map $f:[0,1]\setminus\{c\}\to[0,1]$.
\vspace{0.05cm}

\begin{Lemma}
\label{216}
If $f$ does not have periodic attractors, then 
$$
\alpha_f(x) \ni c \Rightarrow \alpha_f(x) \supset \Omega(f)
$$
\end{Lemma}
\dem
Let $x$ such that $\alpha_f(x) \ni c$ and given $y \in \Omega(f)$ consider any neighborhood $T$ of $y$. As $y$ is non-wandering, there is $z \in T$, (we may assume $z \not \in \co^-_f (c) \cup \co^-_f(Per(f))$) and $j \in \NN$ such that $f^j(z) \in T$. It follows from the homterval lemma that there exists a smallest $t \in \NN$ such that $f^t((z,f^j(z))) \ni c$. 

As $x$ is such that $\alpha_f(x) \ni c$, we have that $\co^-_f (x) \cap f^t(z,f^j(z)) \ne \emptyset$ and, then, $\co^-_f (x) \cap T \supset \co^-_f (x) \cap (z,f^j(z)) \ne \emptyset$. As the chosen neighbourhood $T$ can be taken as small as wanted, we conclude that $y \in \alpha_f(x)$.
 
\cqd

For a Lorenz map $f \notin \cl_{Per}\cup\cl_{Sol} \cup \cl_{Che}$ as in \ref{persolche}, let us define $$\EE = \{x \in (0,1); \alpha_{f}(x) \ni c\}.$$
By Lemma~\ref{vizinhanca} and Proposition~\ref{Proposition008345678}, $\EE$ contains a neighborhood of $c$. In the next lemma, consider $(a,b) \subset \EE$ to be the maximal interval containing $c$.

\begin{Lemma}
$\exists \ell$ and $r>0$ such that $f^\ell((a,c))\subset(a,b)\supset f^r((c,b))$
\end{Lemma}

\dem

As $f((a,c))$ has non-empty interior, it follows from Lemma~\ref{Remark98671oxe} that some iterates of its points will intersect the neighborhood $(a,b)$ of the critical point. Take the minimum $k$ such that $f^k((a,c))\cap (a,b) \ne \emptyset$.

Suppose $f^k((a,c))\not\subset (a,b)$. For example, $b \in  f^k((a,c))$. As $b\in\overline{[0,1]\setminus\EE}$ and $[0,1]\setminus\EE$ is invariant, we get that $f^k((a,c))\cap([0,1]\setminus\EE)\ne\emptyset$ and so, $(a,c)\cap([0,1]\setminus\EE)\ne\emptyset$, which is an absurd.

\cqd

For a Lorenz map $f \notin \cl_{Per}\cup\cl_{Sol} \cup \cl_{Che}$ and $\ell$ and $r$ as given by the former lemma, we define 
\begin{equation}\label{regiaoarmadilha}
\UU = (a,b) \cup \bigg(\bigcup_{j=1}^{\ell-1}f^j((a,c))\bigg) \cup \bigg(\bigcup_{j=1}^{r-1}f^j((c,b))\bigg)\ni c
\end{equation}
and we have that $\UU$ is a trapping region, that is, $f( \UU \setminus \{c\}) \subset \UU$.

It's worth observing that given a non-renormalizable Lorenz map $f$ having a trapping region $\cu$, any point in $(0,1)$ eventually reaches this region when iterated by $f$. Also, the non-wandering set within $(0,1)$ is necessarily inside $\cu$.

\begin{Lemma} Let $f$ be a non-renormalizable Lorenz map  defined in $[0,1]\setminus\{c\}$ and  $\UU \subset [0,1]\setminus\{c\}$ so that $f(\UU)\subset \UU$, then $\forall x \in [0,1]\setminus\{c\}$ $\exists k >0$ such that $f^k(x) \in \UU$.  
\end{Lemma}

\begin{Corollary}
\label{sobrealfa}
For $f \notin \cl_{Per}\cup\cl_{Sol} \cup \cl_{Che}$,  we have that 
$\alpha_{f}(x) \supset \Omega(f)$ $ \forall x \in \UU$.
\end{Corollary}
\dem
As Lemma \ref{216} states that $\alpha_f(x) \supset \Omega(f)$ for any $x$ such that $c \in \alpha_f(x)$, this holds for any $x$ in $\UU$, as this is contained in $\EE$.
\cqd

\begin{Lemma}
\label{0967809hp}
For $f \notin \cl_{Per}\cup\cl_{Sol} \cup \cl_{Che}$, if $\alpha_f(x) \ni c$, then $\alpha_f(x)\cap\UU\subset\Omega(f)\cap\UU$.
\end{Lemma}

\dem
Consider $x$ such that $\alpha_f(x) \ni c$. Given $y \in \alpha_f(x)$, consider any neighborhood $V$ of $y$. We may assume $V \subset \UU$. 

\begin{claim}[A]
$y \in \overline{(V\setminus\{y\})\cap\co^-_f(x)}$.
\end{claim}
\dem
If not, $\exists \epsilon >0$ such that $B_\epsilon(y)\cap\co^-_f(x)=\{y\}$. In this case, we have that $\exists n_1<n_2<...<n_j \to \infty$ such that $f^{n_j}(y)=x$. 
Then, $$x=f^{n_2}(y)=f^{n_2-n_1}(f^{n_1}(y))=f^{n_2-n_1}(x).$$

Observe that if $f^s(B_\epsilon(y))\not\ni c\, \forall s$, then writing $(\alpha, \beta)=f^{n_1}(B_\epsilon(y))$ we have
$$x \in (\alpha,\beta) \text{ and } f^{k(n_2-n_1)}((\alpha,\beta))\not\ni c \forall k.$$
Taking $(x,\gamma)=\bigcup_{k\ge1}f^{k(n_2-n_1)}((x,\beta))=\bigcup_{k\ge1}(x,f^{k(n_2-n_1)}(\beta))$, we have that $f^{n_2-n_1}|_{(x,\gamma)}$ is a homeomorphism and $f^{n_2-n_1}((x,\gamma))\subset(x,\gamma).$

But this would imply the existence of attracting periodic orbits, that are considered not to exist. Then, we necessarily have that $\exists s$ such that $f^s(B_\epsilon(y))\ni c$. 

As $c\in \alpha_{f}(x)$, we would have that $\#\co_f^-(x)\cap B_\epsilon(y)=\infty$. Again a contradiction, proving Claim (A).
\cqd
Because of the Claim we may assume that $y\in \overline{(y,1)\cap V\cap \co^-_f(x)}$ (the proof for the case $y\in$ $\overline{(0,y)\cap V\cap \co^-_f(x)}$ is analogous).

We may take $x_2<x_1\in(y,1)\cap V \cap\co_f^-(x)$ such that $f^{n_2}(x_2)=x=f^{n_1}(x_1)$ with $n_1<n_2$.

\begin{claim}[B]
$\exists s \in \NN$ such that $f^s([x_1,x_2))\ni c$
\end{claim}
\dem
If $c\notin f^s([x_1x_2))$, $\forall\,s\ge0$, then $$f^{k(n_2-n_1)}([f^{n_2-n_1}(x),x))=f^{k(n_2-n_1)+n_2}([x_1,x_2))\not\ni c,\; \forall k \in \NN.$$

As $f$ preserves orientation, $f^{k(n_2-n_1)}|_{[f^{n_2-n_1}(x),x)}$ is a homeomorphism, $\forall x$, $\forall k\ge 0$, so we have $f^{(k+1)(n_2-n_1)}(x)<f^{k(n_2-n_1)}(x) $, $ \forall k\ge 0$.

Then, $\bigcup_{k\ge0}f^{k(n_2-n_1)}([f^{n_2-n_1}(x),x))$ is an interval $(\gamma,x)$. Besides that, $f^{n_2-n_1}|_{(\gamma,x)}$ is a homeomorphism and $f^{n_2-n_1}((\gamma,x))\subset(\gamma,x)$. 

But this is an absurd, because it would imply the existence of attracting periodic orbits, what proves Claim (B).
\cqd

Let $s\in \NN$ such that $f^s([x_1,x_2))\ni c$. As $x_1 \in \UU$, we have that $\co^-_f(x_1)$ accumulates in $c$ by both sides. Then, $\co_f^-(x_1)\cap f^s([x_1,x_2))\ne \emptyset$.

This implies that $\exists x_1' \in \co^-_f(x_1)\cap[x_1,x_2)\subset V$, say $x_1' \in f^{-t}(x_1)\cap V$. Then,
$$
f^t(V)\cap V \ne \emptyset
$$

As $V$ is a neighborhood of $y \in \UU$ that was arbitrarily taken, we may conclude that $y \in \Omega(f)$, proving Lemma \ref{0967809hp}.

\cqd

\begin{Corollary}
\label{poiupoiu} For $f \notin \cl_{Per}\cup\cl_{Sol} \cup \cl_{Che}$,
$\alpha_{f}(x)\cap \UU = \Omega(f)\cap \UU, \forall x \in \UU$.
\end{Corollary}

\begin{Corollary}
\label{compon} For $f \notin \cl_{Per}\cup\cl_{Sol} \cup \cl_{Che}$, then any connected component of $\UU \setminus \Omega(f)$ is a wandering interval.
\end{Corollary}

\dem
Let $J=(a,b)$ connected component of $\UU \setminus \Omega(f)$. Suppose it is not a wandering interval. Then, Lemma~\ref{homtervals} says there will be $n$ for which $f^n(J)\ni c$. Lemma~\ref{vizinhanca} and Proposition~\ref{Proposition008345678} assures us that there are several points with $c$ in their $\alpha$-limits  inside this set $f^n(J)$. We know $f^{-1}(\alpha_{f}(x))\subset\alpha_{f}(x)$ and, then, Corollary~\ref{poiupoiu} assures us these points are in $\Omega(f)$, but they are inside $J$, that should not contain any point of $\Omega(f)$.
\cqd

\begin{Definition}[Strong Transitivity]\label{StTopTrans} Let $\XX$ be a compact metrical space. Given a continuous map $g:A\subset\XX\to\XX$, we say it is strongly (topologically) transitive if for any open set $V\subset\XX$ with $V\cap A\ne\emptyset$, we have $\bigcup_{j \ge 0}g^j(V)=A$.

Let us make precise the notation used in this definition: given $V\subset\XX$, let $g^{-1}(V)=\{x\in A\,;\,g(x)\in V\}$. We define inductively $g^{-n}(V)$, for $n\ge2$, by $g^{-n}(V)=g^{-(n-1)}(g^{-1}(V))$. We define for $n\ge1$, $g^{n}(V)=\{g^{n}(v)\,;\,v\in V\cap g^{-n}(A)\}$. 
\end{Definition}

\begin{Proposition} \label{transitivo} If $f \notin \cl_{Per}\cup\cl_{Sol} \cup \cl_{Che}$, then    
$f|_{\Omega\cap\UU}$ is strongly transitive.
In particular, $$f|_{\Omega\cap\UU}\text{ is transitive.}$$
\end{Proposition}

\dem
We know that $f^{-1}(\alpha_{f}(x))\subset\alpha_{f}(x)$. We will show that $\bigcup_{j\ge0}f^j(V\cap\Omega(t)) = \Omega(f)\cap \UU, \forall V  \subset \UU$, $V$ open and $V \cap \Omega(f) \ne \emptyset$. It follows from the Corollary \ref{poiupoiu} that 
\begin{equation}
\label{inclusaoestrela}
f^{-1}(\Omega(f)\cap\UU)\cap\UU\subset\Omega(f)\cap\UU.
\end{equation}

Let $V\subset\UU$, $V$ any open set with $V \cap \Omega(f) \ne \emptyset$. Given $x \in \Omega(f)\cap\UU$, we have that $\alpha_{f}(x)\cap V\ne \emptyset$ and, then, $\co^-_f(x)\cap V \ne \emptyset$. Pick $x_t \in f^{-t}(x) \cap V$. Define $x_k=f^{t-k}(x_t)$ for $0 \le k \le t$. 
$$
x_t \stackrel{f}{\to} x_{t-1} \stackrel{f}{\to} \dots \stackrel{f}{\to} x_0=f^t(x_t) 
$$
As $\UU$ is a trapping region, we have that $x_k$ in $\UU, \forall 0 \le k \le t$. 

We claim that $x_t \in \Omega(f)\cap\UU$. Indeed, we have that $x_0 \in \Omega(f)\cap\UU$. Suppose it also works for $k-1$, that is, $x_{k-1} \in \Omega(f)\cap\UU$. We have that $x_k \in \UU$. Then $x_k \in f^{-1}(x_{k-1})\cap \UU$ and by (\ref{inclusaoestrela}) we have that $x_k \in \Omega(f)\cap \UU$.
It follows by induction that $x_t \in \Omega(f)\cap\UU$.

\cqd

\begin{Theorem} 
\label{cicloint}
Let $f:[0,1]\setminus\{c\}\to[0,1]$ be a $C^{2}$ non-flat contracting Lorenz map. 
If $f$ doesn't have a periodic attracting orbit, isn't a Cherry map nor $\infty$-renormalizable, then there is an open trapping region $U\ni c$ given by a finite union of open intervals such that $\Lambda:=\overline{U\cap\Omega(f)}$ satisfies the following statements.

\begin{enumerate}

\item$\omega_{f}(x)=\Lambda$ for a residual set of points of $\Lambda$ (in particular, $\Lambda$ is transitive).  
\item The basin of attraction of $\Lambda$, $\beta(\Lambda):=\{x;\omega_{f}(x)\subset \Lambda\}$, is an open and dense set.
\item $\exists \lambda >0$ such that $\lim_{n\to\infty} \frac{1}{n}\log |Df^n(x)|=\lambda$ for a dense set of points $x$ in $\Lambda$.
\item \label{opcoes} either $\Lambda$ is a finite union of intervals or it is a Cantor set.
\item if $\Lambda$ is a finite union of intervals, then $\omega_{f}(x)=\Lambda$ for a residual set of $x$ in $[0,1]$.
\item \label{cantor} $\Lambda$ is a Cantor set if and only if there is a wandering interval.

\end{enumerate}
\end{Theorem}

\dem Set $\Lambda:=\overline{\Omega(f) \cap\UU}$ with $\UU$ as defined in (\ref{regiaoarmadilha}).
\begin{enumerate}

\item Lemma~\ref{dicotomia} of Appendix insures us it is true, as we have transitivity provided by Proposition~\ref{transitivo}.
\item 

By Lemma~\ref{Remark98671oxe}, the set $\cu = \{ x \in [0,1]\setminus\{c\}; \exists j \text{ such that } f^j(x)\in \UU\}$ is an open and dense set. We claim that any point $y$ in this set $\cu$ is also in $\beta(\Lambda)$. For some $k$, $f^k(y)=x \in \UU$, and we have two possible situations for a point $q\in\omega_{f}(x)=\omega_f(y)$. As $\UU$ is a trapping region, $q$ can be an interior point of $\UU$, and then it automatically belongs to $\Lambda=\overline{\Omega \cup \UU}$. If not an interior point, $q\in\partial\UU$. In this case, as $q\in\omega_{f}(x)$, there are infinitely many $f^{n_j}(x)$ accumulating in $q$. Then, there can be no wandering interval with border $q$ (as images of $x$ keep coming close to $q$). By Corollary~\ref{compon}, as $q$ can't be in the border of a wandering interval, it is not in the border of a connected component of $\UU\setminus\Omega(f)$, then it is accumulated by points of this set, that is, $q\in \Lambda=\overline{\Omega(f) \cap\UU}$.

\item Proposition (\ref{Proposition008345678}) says that repeller points $p \in Per(f)$ accumulate in $c$. As they are in $\Omega(f)$, the ones that are in $\UU$ are also in $\Lambda$, and it follows from Corollary(\ref{poiupoiu}) that $\co^-_f(p)$ is dense in $\Lambda$. 
Given any point $x\in\co^-_f(p)$, as it is eventually periodic, say $f^j(x)=p$ (and as there are infinitely many ones, we can pick one such that $c$ is not in its pre-orbit, in order to proceed with the following computation), we have
$$\lim_{n\to\infty} \frac{1}{n}\log |D(f^{n-j}\circ f^j)(x)|=\lim_{n\to\infty} \frac{1}{n}\log\big(|Df^{n-j}(p)|\big)+\lim_{n\to\infty} \frac{1}{n}\log|Df^j(x)|=$$
$$=\lim_{n\to\infty} \frac{n-j}{n(n-j)}\log\big(|Df^{n-j}(p)|\big)=\lim_{n\to\infty} \frac{1}{n-j}\log\big(|Df^{n-j}(p)|\big)=\exp_f(p)=:\lambda.$$

\item \label{item4} As $\Lambda$ is transitive, $\exists x \in \Lambda=\omega_{f}(x)$, then, by Lemma~\ref{aeroporto} of Appendix, it is a perfect set.
We have two possibilities: $\interior(\Lambda)= \emptyset$ or not.
As $\Lambda$ is a subset of $\RR$, if it has empty interior, it is totally disconnected. Consequently, it will be a Cantor set (as we already proved it is compact and perfect).
Suppose, then, $\interior(\Lambda) \ne \emptyset$. Let $I$ be an open interval, $I \subset \Lambda$ and it can't be a wandering interval, as it is a subset of $\Lambda \subset \Omega(f)$. Then, by Lemma~\ref{homtervals}, $\exists j$ such that $f^{j}(I)\ni c$, and so, $c \in \interior\Lambda$. This forbids the existence of wandering intervals. Indeed, if there is a wandering interval $J$, it has to accumulate in the critical point (by Lemma~\ref{LemmaWI}), but this would imply that $f^{n}(J)\cap\Omega(f)\ne\emptyset$ for $n$ sufficiently big. An absurd. So, as we cannot have wandering intervals, Corollary~\ref{compon}, $\UU\setminus\Omega(f)$ has to be an empty set. As $\UU$ is an orbit of intervals, it proves the claim of the Theorem.
\item Let $\Lambda'=\{x\in\UU;\omega_{f}(x)=\Lambda\}$.
Observe that $x \in \bigcup_{j\ge0}f^{-j}(\Lambda')$ implies that $\omega_{f}(x)=\Lambda$.
As $\Lambda'$ is residual in $\UU$, there exist $A_n$, $n\in\NN$, open and dense sets in $\UU$ such that $\Lambda'=\bigcap_{n\in \NN}A_n$.
On the other hand, for every $n\in\NN$ we have that $\bigcup_{j\ge0}f^{-j}(A_n)$ is an open dense set in $[0,1]$.
Then, $\bigcap_{n\in\NN}\big(\bigcup_{j\ge0}f^{-j}(A_n)\big)$ is residual in $[0,1]$.
So we have that $\bigcup_{j\ge0}f^{-j}\big(\Lambda'\big)=\bigcup_{j\ge0}f^{-j}\big(\bigcap_{n\in\NN}A_n\big)=\bigcap_{n\in\NN}\big(\bigcup_{j\ge0}f^{-j}(A_n)\big)$ is residual.

\item It follows straightforwardly from the former construction: $\Lambda$ being a Cantor set implies that $\UU\setminus\Omega(f)$ has non-trivial connected components, that Lemma~\ref{homtervals} says it is a wandering interval. The converse, for as $\Lambda$ is compact and perfect, if we suppose $\interior(\Lambda)\ne\emptyset$, following the same reasoning of (\ref{item4}), there would be an interval $I$ such that $f^j(I)\ni c$ for some $j$, contradicting the existence of wandering interval.

\end{enumerate}

\cqd

\begin{Lemma} 
\label{perdenso} Let $f:[0,1]\setminus\{c\}\to[0,1]$ be a $C^{2}$ non-flat contracting Lorenz map. 
Suppose that $f \notin \cl_{Per}\cup\cl_{Sol} \cup \cl_{Che}$. Then $\overline {Per(f)\cap \Lambda}=\Lambda$, with $\Lambda$ as obtained in Theorem \ref{cicloint}.

\end{Lemma}
\dem
Notice that $ {Per(f)\cap \overline\UU}= Per(f)\cap \Lambda$, thus $\Lambda\setminus\overline {Per(f)}=\Lambda\setminus\overline {Per(f)\cap \Lambda}$. Suppose that $\Lambda\setminus\overline {Per(f)}\ne\emptyset$. Let $I$ be connected component of $\UU\setminus\overline {Per(f)}$ such that $I\cap\Lambda\ne\emptyset$. As $\Lambda$ is perfect and compact we have that $I\cap\Lambda$ is uncountable. Moreover, as $\{x\in\Lambda;\omega_f(x)=\Lambda\}$ is residual in $\Lambda$, we have that $\{x\in\Lambda;\omega_f(x)=\Lambda\}\cap I$ is uncountable.
Then, the set of points that return infinitely many times to $I$ (that is, $\bigcap_{j\ge0}f^{-j}(I))$ is uncountable.
Let $I^*=\{x\in I;\co^+_f(f(x))\cap I\ne\emptyset\}$ be the set of points that return to $I$ and $F:I^*\to I$ the first return map. 
Observe that the set of points that return infinitely many times to $I$ is given by
$$
\{x; \#(\co^+_f(x)\cap I)=\infty\}=\bigcap_{j\ge0}F^{-j}(I)
$$
This way
\begin{equation}
\label{estdavi}
\bigcap_{j\ge0}F^{-j}(I) \text{ is uncountable.} 
\end{equation}

\begin{claim}[a]
If $J$ is connected component of $I^*$, then $F(J)=I$.
\end{claim}
{\em Proof of the Claim.}
Let $I=(i_0,i_1)$. If $F(J)\ne I$, then let $(t_0,t_1)=F(J)$ and in this case $t_0\ne i_0$ or $t_1\ne i_1$. Suppose $t_0\ne i_0$ (the other case is analogous). Let $n=R(J)$.
As $t_0\ne i_0$, there is $0\le s< n$ such that $f^s(t_0)=c$. Then we have that
$$
\#\big(Per(f)\cap f^s(J)\big)=\#\big(Per(f)\cap (c,f^s(t_1))\big)=\infty,
$$
as the periodic points accumulate in both sides of the critical point (Proposition~\ref{Proposition008345678}).
Then $\#\big(Per(f)\cap I\big)\ge\#\big(Per(f)\cap f^n(J)\big)=\infty$, contradicting the fact that $I$ is connected component of $\UU\setminus\overline {Per(f)}$.
$\square $(end of the proof of the Claim (a))  

\begin{claim}[b]
$I^*$ has more than one connected component.
\end{claim}
{\em Proof of the Claim.}
Suppose it isn't so, then $I^*$ is an interval and we will write it as $(u,v)$ and $F=f^n|{(u,v)}$ for some $n\in\NN$. This implies, then, that $\bigcap_{j\ge0}F^{-j}(I)=Fix(f^n|{(u,v)})$. 
But this is an absurd, as by equation (\ref{estdavi}) this set would be uncountable and so the set of periodic points of $f$ would also be uncountable.

$\square $(end of the proof of the Claim (b))

\begin{figure}
\begin{center}\label{PerDenso.png}
\includegraphics[scale=.27]{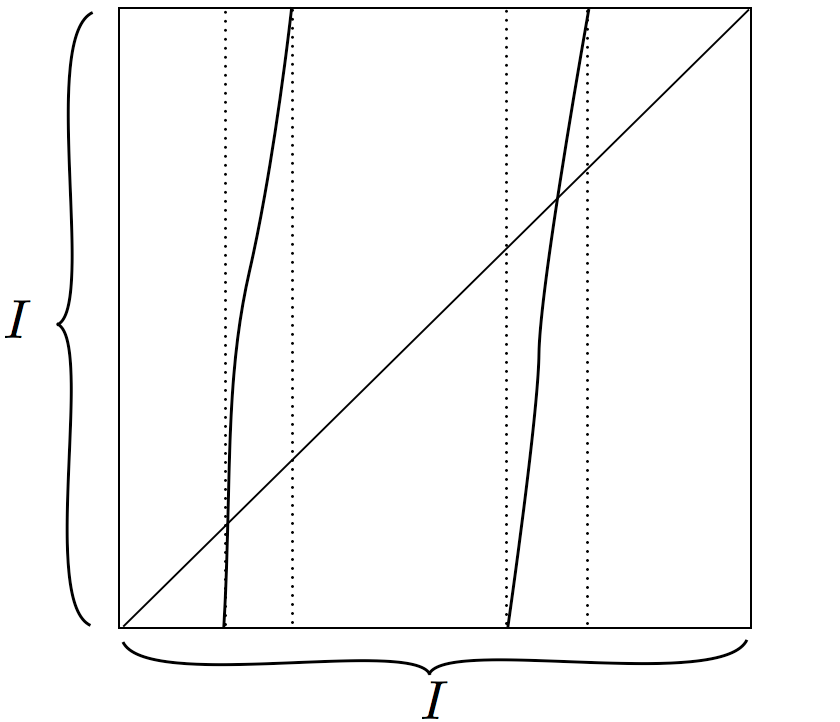}\\
Figure~\ref{PerDenso.png}
\end{center}
\end{figure}

As $F$ has at least two branches covering the full image $I$, we have it has infinitely many periodic points and, then, $f$ also has infinitely many periodic points in $I$, absurd.

\cqd

\section{Proof of Theorems \ref{baciastopologicas}, \ref{teoalfalim} and \ref{atratortopologico}}\label{ProofABCD}

Now, we will prove the main theorems: Theorem \ref{baciastopologicas}, \ref{teoalfalim} and \ref{atratortopologico}. 

\dem[Proof of Theorem~\ref{baciastopologicas}]
We are supposing $f$ has no attracting periodic orbit. Besides that, let's consider different situations:

\begin{enumerate}

\item Firstly, let us suppose that $\exists \varepsilon>0$ such that $B_\varepsilon(c)\cap Per(f)=\emptyset$. Then $[0,1]\setminus \overline{Per(f)}$ has a connected component $J=(a,b)$ such that  $c\in J$. 

If $\exists n$ such that $f^n(a)\in J$, $\exists \varepsilon >0$ such that $f^n(B_\varepsilon(a))\subset J$.
As $Per(f) \cap B_\varepsilon(a)\neq \emptyset$, then $Per(f)\cap J\neq\emptyset$, in contradiction with the definition of $J$.
Similarly we show that $f^j(b)\not\in J$, $\forall j \in \NN$, and so $J$ is a nice interval.
Lemma \ref{LemmaHGFGH54} states that $a\in Per(f)$ or it is accumulated by periodic points $p_j \in J$, and the same for $b$. Then, $\{a,b\}\subset Per(f)$.

We can also state that $J$ is a renormalization interval, for if $f^{period(a)}((a,c))\not\subset(a,b)$, by Lemma \ref{Lemma8388881a}, $\exists d \in (a,b)$ such that $f^{period(a)}((a,d))=(a,b)$, that is, $f^{period(a)}(d)=(b)$ and then $(d,b)$ is nice, but again by lemma \ref{LemmaHGFGH54}, $d \in Per(f)$ or $\exists p_j \in Per(f)$, $p_j \nearrow d$, which is a contradiction. In the same way, $f((c,b))\subset(a,b)$ and, so, $J$ is a renormalization interval.

As there are no attracting periodic orbits and $J$ is a renormalization interval, it follows from Lemma~\ref{noitenoite} that $\omega_{f}(x)\ni c, \,\forall x \in J$. By a renormalization and Lemma~\ref{atcher} in the Appendix, there is a compact minimal set $\Lambda$ such that $\omega_{f}(x)=\Lambda$, $\forall\,x\in J$. Then this is a Cherry map, according to the equivalency provided by \cite{GT85}, as observed when we defined Cherry maps. Also, as Lemma \ref{Remark98671oxe} assures us that $\{x\,;\,\co_{f}^{+}(x)\cap J\ne\emptyset\}$ is an open and dense set, it is not difficult to conclude that $\Lambda$ is a Cherry attractor, and that it attracts a residual set of the interval.

One can observe that all these features of the Cherry attractor could also be obtained using the semi-conjugation with an irrational rotation.

It may occur that the semi-conjugacy is not surjective, meaning the Cherry map has a gap, that is, there is a wandering interval for the considered map. 

For the remaining cases we have, then, that $\forall \varepsilon>0 \,\exists  p; p\in B_\varepsilon(c)\cap Per(f)$. Among these, the first situation to consider is the one of $\Lambda$ being a solenoidal attractor:

\item 
As we have defined, there is a set $\Lambda \subset \bigcap_{n=0}^\infty K_n$, where $K_n=\bigcup_{j=0}^{\period(p_{n})}f^{j}((p_n,c))\cup \bigcup_{j=0}^{\period(q_n)}f^{j}((c,q_n))$, $J_n=(p_{n},q_{n})$, $n\in\NN$, and $J_1 \supset J_2 \supset \cdots $ is the chain of renormalization intervals.

It follows from the construction that $c \in \Lambda$. Moreover, it follows from Lemma \ref{Remark98671oxe} that given a renormalization interval $J_n$, the set of points that eventually visit it is an open and dense set, $V_n=\{x; \exists j \text{ such that }f_j(x)\in J_n\}$. There is a residual set $\bigcap_{n=0}^{\infty}V_{n}$ of points that eventually fall into any renormalization interval, that is, $c \in \omega_f(x)$, $\forall x \in \bigcap_{n=0}^{\infty}V_{n}$ and by theorem \ref{SOLENOIDETH}, $\omega_f(x)=\Lambda$, that is, this residual set belongs to the basin of $\Lambda$, as stated.

\item Now we come to the situation that $f$ has no periodic attractor, neither Cherry attractor nor Solenoidal attractor. It follows from Theorem~\ref{cicloint} that $\exists \Lambda$ compact, $f(\Lambda)=\Lambda$, transitive set   such that $\omega_{f}(x)=\Lambda$ for a residual set of points of $\Lambda$, whose basin of attraction $\beta(\Lambda):=\{x;\omega_{f}(x)\subset \Lambda\}$, is an open and dense set. Also, $\exists \lambda >0$ such that $\lim_{n\to\infty} \frac{1}{n}\log |Df^n(x)|=\lambda$ for a dense set of points $x$ in $\Lambda$.

Theorem~\ref{cicloint} also gives two possibilities for this setting:
\begin{enumerate}

\item either $\Lambda$ is a finite union of intervals and $\omega_{f}(x)=\Lambda$ for a residual set of $x$ in $[0,1]$
\item or it is a Cantor set and there is a wandering interval.
\end{enumerate}

In both cases, all we have to do to complete the proof of the theorem is to show that any of these two is a chaotic attractor, and for this, it only remains to prove that periodic orbits are dense in it ($\overline {Per(f)\cap \Lambda}=\Lambda$) and that its topological entropy $h_{top}(f|_\Lambda)$ is positive. The condition on the periodic points follows from Lemma~\ref{perdenso}. 
The fact that the topological entropy is positive can be obtained by taking arbitrarily small nice intervals whose borders are non-periodic (e.g., pre-periodic points), and by observing that the returns to this interval provide at least two full branches, that will create shifts that have positive entropy.

\end{enumerate}

\cqd

\dem[Proof of Theorem~\ref{teoalfalim}]
The existence of a single topological attractor is given by Theorem~\ref{cicloint}. If $\Lambda$ is a Cherry attractor and it does not have a wandering interval, then there is an interval $[a,b]$, such that (identifying $a$ and $b$) the first return map to $F:[a,b]\to[a,b]$ is conjugated to an irrational rotation. In particular $\alpha_{f}(x)\supset\alpha_{F}(x)=[a,b]=\omega_{F}(x)\subset\omega_{f}(x)$, $\forall\,x\in[a,b]$. Furthermore, the attractor for the map $f$, $\Lambda$, is given by the itinerary of the interval $[a,b]$, that is, $\Lambda=[a,b]\cup \bigcup_{j=0}^{\ell-1}f^{j}([f(a),f(c_{-})])\cup \bigcup_{j=0}^{r-1}f^{j}([f(c_{+}),f(b)])$, where $\ell$ and $r$ are the smallest integers such that $f^{\ell}((a,c))\subset(a,b)\supset f^{r}((c,b))$. So, $$\alpha_{f}(x)\supset\Lambda\subset\omega_{f}(x)\,\,\forall\,x\in\Lambda.$$ In particular, $$\alpha_{f}(x)\supset(a,b)\subset\omega_{f}(x)\,\,\forall\,x\in\Lambda.$$

Considering $V_{(a,b)}=\{x\in[0,1]\,;\, \exists j \text{ such that }f^j(x)\in(a,b)  \}$, Lemma \ref{Remark98671oxe} assures us that this set is open and dense, and then we get $\alpha_{f}(x)=[0,1]$, $\forall\,x\in\Lambda$.

If $\Lambda$ is a Solenoid, then $\Lambda\subset\bigcap_{n=0}^\infty K_n$, where $$K_n=\bigg(\bigcup_{j=0}^{\period(a_{n})}f^j([a_{n},c))\bigg)\cap\bigg(\bigcup_{j=0}^{\period(b_{n})}f^{j}((c,b_{n}])\bigg)$$ and $\{J_n=(a_{n},b_{n})\}_{n}$ is an infinite nested chain of renormalization intervals. Given any $x\in K_{J_{n}}$ and  $y \in \Lambda_{J_{n}}$, there are $w\in J_n$ and $\ell \in \NN$ such that $f^\ell(w)=x$. By Lemma \ref{Lemma549164}, for any given $\varepsilon>0$, $\exists z \ \in B_\varepsilon(y)$ such that $f^k(z)=x$ for some $k>0$. Then, $\alpha_{f}(x)\supset \Lambda_{J_{n}}$, $\forall n \in \NN$. If $f$ does not have any wandering interval, it is easy to show that $\bigcup_{n\ge0}\Lambda_{J_{n}}$ is dense in $[0,1]$: suppose it isn't, then $\exists U$ open interval $U\in [0,1]\setminus \bigcup_{n\ge0}\Lambda_{J_{n}}$. If $\exists j$ such that $f^j(U)\ni c$, take $j$ minimum with this property. As $f^j(U)$ is an open neighborhood of $c$, then it contains $\overline{J_m}$ for $m$ big enough, where $J_m$ is a renormalization interval. Then, $\exists s$ and $t\in U$ such that $f^j(s)=a_m$ and $f^j(t)=b_m$, which is in contradiction with the definition of $U$, as $\{s,t\}\subset \Lambda_{J_{m}}$. As $\alpha_{f}(x)\supset \Lambda_{J_{n}}$ and $\bigcup_{n\ge0}\Lambda_{J_{n}}$ is dense in $[0,1]$, we have proved that $\alpha_{f}(x)=[0,1]$, $\forall\,x\in\Lambda$.

Finally, if $\Lambda$ is not a Cherry or a Solenoid attractor, the proof follows from Corollary~\ref{sobrealfa} and items (\ref{opcoes}) and  (\ref{cantor}) of Theorem~\ref{baciastopologicas}. Indeed, as we are assuming that $f$ does not have wandering intervals, it follows from items (\ref{opcoes}) and (\ref{cantor}) of Theorem~\ref{baciastopologicas} that $\Lambda$ is a cycle of intervals. By Corollary~\ref{sobrealfa} and the fact that $\Lambda=\overline{\UU\cap\Omega(f)}$, we get $\alpha_{f}(x)\supset\Lambda$. As $\Lambda$ contains an open neighborhood of $c$, it follows that the set of $x\in[0,1]$ such that $\co_{f}^{+}(x)\cap\Lambda\ne\emptyset$ contains an open and dense set. Thus, $\bigcup_{j\ge0}f^{-j}(\Lambda)$ is dense and so $\alpha_{f}(x)$ is dense $\forall\,x\in\UU$. As the $\alpha$-limit is a closed set, $\alpha_{f}(x)=[0,1]$ for all $x\in\UU$. If $\Lambda=\overline{\UU\cap\Omega(f)}=\UU\cap\Omega(f)\subset\UU$, the proof is done. On the other hand, if $\Lambda\not\subset\UU$, then $\Lambda\setminus\UU\subset(\co_{f}^{+}(c_{-})\cap\co_{f}^{+}(c_{+}))$. But, as it was defined in the beginning of Section~\ref{MainResults}, $c\in f^{-1}(f(c_{-}))$ and also $c\in f^{-1}(f(c_{+}))$. Thus, $\alpha_{f}(\Lambda\setminus\UU)\supset\alpha_{f}(c)=[0,1]$ (because $c\in\UU$).
\cqd

\dem[Proof of Theorem~\ref{atratortopologico}]
The first statement of the theorem follows straightforwardly from Proposition~\ref{DOISMAX}. 
Items (1),(2) and (3) repeat what is said in Theorem~\ref{baciastopologicas}. 
In the case (4), we have the existence of wandering intervals, so let's consider $V$ the union of all wandering intervals. Lemma~\ref{DenWanInt} says this set is open and dense in $[0,1]$, and Corollary~\ref{stpremovido} gives the structure of the set $\Lambda$. 
\cqd
\color{black}

\section{Appendix}

\begin{Lemma}\label{dicotomia} 
If $f:U \to \XX$ is a continuous map  defined in an open and dense subset $U$ of  compact metric space $\XX$, then either $\nexists x \in U$ such that $\omega_{f}(x)=\XX$ or $\omega (x)=\XX$ for a residual set of $x \in \XX$.
\end{Lemma}

\dem
Suppose that $\co^{+}_{f}(p)$ is dense in $\XX$ for some $p\in \bigcap_{j\ge0}f^{-j}(U)$. Write $p_{\ell}=f^{\ell}(p)$. For each $\ell\in\NN$ there is some $k_{n,\ell}$ such that $\{p_{\ell},$ $\cdots,$ $f^{k_{n,\ell}}(p_{\ell})\}$ is $(1/2n)$-dense. As $f$ is continuous and $U$ open, there is some $r_{n,\ell}>0$ such that $f^{j}(B_{r_{n,\ell}}(p_{\ell}))\subset B_{1/2n}(f^{j}(p_{\ell}))$, $\forall\,0\le j\le k_{n,\ell}$. Thus, $\{y,\cdots,f^{k_{n,\ell}}(y)\}$ is $(1/n)$-dense $\forall\,y\in B_{r_{n,\ell}}(p_{\ell})$. Let $$\XX_{n}=\{x\in\XX\,;\,\co^{+}_{f}(x)\text{ is }(1/n)-\text{dense}\}.$$ Therefore $\bigcup_{\ell\in\NN}B_{r_{n,\ell}}(p_{\ell})\subset\XX_{n}$ is a open and dense set. Furthermore, $$\bigcap_{n\in\NN}\bigcup_{\ell\in\NN}B_{r_{n,\ell}}(p_{\ell})$$ is a residual set contained in $\bigcap_{n\in\NN}\XX_{n}=\{x\in\XX$ $;$ $\omega_{f}(x)=\XX\}$.
\cqd

\begin{Lemma}
\label{aeroporto}
Let $\XX$ be a compact metric space and $f:U\to\XX$ be a continuous map defined in a subset $U$. If $x\in\bigcap_{n\ge0}f^{-n}(U)$ and $x\in\omega_f(x)$, then either $\co_f^+(x)$ is a periodic orbit (in this case $\omega_f(x)=\co_f^+(x)$) or $\omega_{f}(x)$ is a perfect set.
\end{Lemma}

\dem
Suppose $\exists p \in \omega_{f}(x)$ an isolated point, say $B_{\varepsilon}(p)\cap\omega_{f}(x)=\{p\}$, with $\varepsilon>0$. As $x\in\omega_{f}(x)$ and $f$ is continuous on $\co^{+}_{f}(x)$, we have $\co^{+}_{f}(x)\subset\omega_{f}(x)$. Thus, $\co^{+}_{f}(x)\cap(B_{\varepsilon}(p)\setminus\{p\})=\emptyset$.
As $p \in \omega_{f}(x) \Rightarrow \exists$ sequence $n_{j}\nearrow \infty$ such that $f^{n_{j}}(x)\to p$.  Taking $j$ big enough we have $f^{n_{j}}(x)\in B_{\varepsilon}(p)$, then $f^{n_{j}}(x)=p$, $\forall j$ big and, then, $f^{n_{j+1}-n_{j}}(f^{n_{j}}(x))=p=f^{n_{j}}(x)$, that is, $f^{n_{j}}(x)$ is periodic. As $x \in \omega_{f}(x)=\omega(f^{n_{j}}(x))=\co^{+}(f^{n_{j}}(x))$, we have that $x$ is periodic.
\cqd

\begin{Corollary}
\label{CORaeroporto}
Let $f:[0,1]\setminus\{c\}\to[0,1]$ be a contracting Lorenz map. If $c_-\in\omega_f(c_-)$, then either $f$ has a super-attractor containing $c_-$ or $\omega_f(c_-)$ is a perfect set. Analogously,   If $c_+\in\omega_f(c_+)$, then either $f$ has a super-attractor containing $c_+$ or $\omega_f(c_+)$ is a perfect set.
\end{Corollary}

\dem
Suppose that $f$ does not have a super-attractor containing $c_-$. Thus, $v_1:=f(c_-)\notin\co_f^-(c)$. In this case, $\co_f^+(c_-)=\{c\}\cup\co_f^+(v_1)$ (recall the definition of $\co_f^+(c_-)$ in the beginning of Section~\ref{MainResults}). Note that $v_1\in\omega_f(v_1)$, because $c_-\in\omega_f(c_-)$. As $v_1$ can not be a periodic orbit and as $v_1\in\bigcup_{n\ge0}f^{-n}([0,1]\setminus\{c\})$, it follows from Lemma~\ref{aeroporto} that $\omega_f(v_1)$ is a perfect set. As $\omega_f(c_-)=\omega_f(v_1)$  (because $c\in\overline{\omega_f(v_1)\cap(0,c)}$), we finish the proof.\cqd

\begin{Corollary}
\label{CORaeroporto2}
Let $f:[0,1]\setminus\{c\}\to[0,1]$ be a contracting Lorenz map without periodic attractors. Suppose $\omega_f(c_-)\ni c\in\omega_{f}(c_{+})$. If $\overline{\co_{f}^{+}(p)\cap(0,c)}\ni c\in\overline{(c,1)\cap\co_{f}^{+}(p)}$, $p\in(0,1)\setminus\{c\}$, then $\omega_{f}(p)$ is a perfect set and $\overline{\omega_{f}(p)\cap(0,c)}\ni c\in\overline{(c,1)\cap\omega_{f}(p)}$.\end{Corollary}
\begin{proof}
It follows from Corollary~\ref{CORaeroporto} that $\omega_f(c_-)$ and $\omega_f(c_+)$ are perfect sets. Furthermore, $\overline{\omega_f(c_-)\cap(0,c)}\ni c\in\overline{(c,1)\cap\omega_f(c_+)}$.
If $\overline{\co_{f}^{+}(p)\cap(0,c)}\ni c\in\overline{(c,1)\cap\co_{f}^{+}(p)}$
then $\omega_f(p)\supset\omega_f(c_-)\cup\omega_f(c_+)$ and so,
$\overline{\omega_f(p)\cap(0,c)}\ni c\in\overline{(c,1)\cap\omega_f(p)}$.

Now suppose that $\omega_{f}(p)$ is not perfect. Thus, there is $q\in\omega_{f}(p)$ and $\delta>0$ such that $B_{\delta}(q)\cap\omega_{f}(p)=\{q\}$. Let $J=(a,b)$ be the connect component of $[0,1]\setminus\big(\omega_{f}(p)\setminus\{q\}\big)$ containing $q$. Note that $a,b\subset\big(\omega_{f}(x)\cup\{0,1\})$.

\end{proof}

\color{black}

\subsubsection{The attractor for Cherry maps}\label{SecAtCherry}

\begin{Lemma}\label{atcher}
Let $f:[0,1]\setminus\{c\}\to[0,1]$ be a contracting Lorenz map without super-attractors. If $c\in\omega_f(x)$, $\forall\,x\in(0,1)$, then there exists a compact set $\Lambda\subset(0,1)$ such that $\omega_f(x)=\Lambda$, $\forall\,x\in(0,1)$. In particular, $\Lambda$ is a minimal set.
\end{Lemma}
\begin{proof}

As $f$ does not have super-attractor and $c\in\omega_f(x)$, $\forall\,x\in(0,1)$, we get $$Per(f)=\{0,1\}.$$
Note also that $f([0,c))\ni c\in f((c,1])$, because $c\in\omega_f(x)$, $\forall\,x\in(0,1)$.
Taking in Lemma~\ref{Lemma545g55} $(a,b)=(0,1)$, we conclude that $\co_f^+(x)\cap(0,c)\ne\emptyset\ne(c,1)\cap\co_f^+(x)$, $\forall\,x\in(0,1)$.
So, by Lemma \ref{omegaemc} we get
\begin{equation}\label{EqCL}
\overline{\co_f^+(x)\cap(0,c)}\ni c\in \overline{(c,1)\cap\co_f^+(x)},\,\;\forall\,x\in(0,1).
\end{equation}
As a consequence,
\begin{equation}\label{Eq323130}\omega_f(x)\supset\omega_f(c_-)\cup\omega_f(c_+),\,\;\forall\,x\in(0,1).
\end{equation}
In particular, $$c_-\in\omega_f(c_-)\text{ and }c_+\in\omega_f(c_+).$$

Thus, it follows from Corollary~\ref{CORaeroporto2} that
\begin{equation}\label{EqUU21}
\overline{\omega_f(x)\cap(0,c)}\ni c\in\overline{(c,1)\cap\omega_f(x)}\,\,\;\forall\,x\in(0,1).
\end{equation}

Now we will prove that $\omega_f(p)=\omega_f(q)$, $\forall\,p,q\in(0,1)$. If this is not true, then there exist $p,q\in(0,1)$ such that $\omega_f(p)\setminus\omega_f(q)\ne\emptyset$. Let $y\in \omega_f(p)\setminus\omega_f(q)$. Set $[\alpha,\beta]=\big[\min\omega_f(p),\max\omega_f(p)\big]$ (indeed, $[\alpha,\beta]=[f(c_{+}),f(c_{-})]$).  It is easy to see that $f([\alpha,\beta])=[\alpha,\beta]$. As $c\in (\alpha,\beta)$ (by (\ref{EqCL})) and as $c\in\omega_f(x)$, $\forall\,x$, we get $\omega_f(x)\subset[\alpha,\beta]$, $\forall\,x\in(0,1)$. As a consequence, $y\in(\alpha,\beta)$.

Let $J=(a,b)$ be the connected component of $[0,1]\setminus\omega_f(p)$ containing $q$. As $y\in(\alpha,\beta)$, we get $a,b\in\omega_f(p)$. As $y\in \omega_f(q)\cap J$, one can find $0\le n_1<n_2$ such that $f^{n_1}(q),f^{n_2}(q)\in(a,b)$. We may suppose that $f^{n_1}(q)<f^{n_2}(q)$ (the case $f^{n_1}(q)>f^{n_2}(q)$ is analogous).

Let $T:=(t,f^{n_1}(q)]$ be the maximal interval contained in $(a,f^{n_1}(q)]$ such that $f^{n_2-n_1}|_T$ is a homeomorphism and that $f^{n_2-n_1}(T)\subset (a,f^{n_2}(q)]$.

\begin{Claim}
$f^{n_2-n_1}(T)=(a,f^{n_2}(q)]$
\end{Claim}
\begin{proof}[Proof of the claim]
If not, there are two possibles cases: (1) $f^{s}(t)=c$ for some $0\le s<n_2-n_1$ or (2) $t=a$ and $a<f^{n_2-n_1}(a)<f^{n_{2}}(q)$. As $a<f^{n_{2}-n_{1}}(a)<f^{n_{2}}(q)$ will implies that $\omega_{f}(p)\cap J\ne\emptyset$, and this contradicts the fact that $J\subset[0,1]\setminus\omega_{f}(p)$, we have only to analyze the first case.

Thus, $f^s(T)\cap\omega_f(p)=(c,f^s(f^{n_2-n1})(q))\cap\omega_f(p)\ne\emptyset$ (because of (\ref{EqUU21})). But this implies that  $J\cap\omega_f(x)\supset f^{n_2-n_1}(T)\cap\omega_f(x)\supset f^{n_2-n_1-s}(f^s(T)\cap\omega_f(x))\ne\emptyset$. An absurd, as $J\subset[0,1]\setminus\omega_f(x)$.\end{proof}

It follows from the claim above that $f^{n_2-n_1}(T)=(a,f^{n_2}(q)]\supset T$. This implies that $f$ has a periodic point in $\overline{T}$ (because $f^{n_2-n_1}|_{T}$ is a homeomorphism). But this is a contradiction with the fact that $Per(f)\cap(0,1)=\emptyset$.
\end{proof}

\end{document}